\documentclass[11pt]{amsart}

\usepackage{epigamath}

\usepackage[english]{babel}

\numberwithin{equation}{section}

\usepackage{tabu}
\usepackage{environ}
\usepackage{mathtools}
\usepackage[all]{xy}
\usepackage{tikz}
\usepackage{tikz-cd}
\usepackage{stmaryrd}
\usepackage{xfrac}
\usepackage{fix-cm}
\usepackage{faktor}
\usepackage{minibox}

\usepackage{cleveref}

\newtheorem{theorem_intro}{Theorem}

\newtheorem{theorem}{Theorem}[section]

\newtheorem{lemma}[theorem]{Lemma}
\newtheorem{proposition}[theorem]{Proposition}
\newtheorem{claim}[theorem]{Claim}

\theoremstyle{definition}

\newtheorem{definition}[theorem]{Definition}

\newtheorem{assumption}[theorem]{Assumption}

\theoremstyle{remark}

\newtheorem{examples}[theorem]{Examples}
\newtheorem{remark}[theorem]{Remark}
\newtheorem{observation}[theorem]{Observation}

\newcommand{\fI}{\mathfrak{I}}

\newcommand{\bA}{\mathbb{A}}
\newcommand{\bF}{\mathbb{F}}
\newcommand{\bP}{\mathbb{P}}
\newcommand{\bQ}{\mathbb{Q}}
\newcommand{\bZ}{\mathbb{Z}}

\newcommand{\scr}{\mathcal}
\newcommand{\sF}{\scr{F}}
\newcommand{\sO}{\scr{O}}

\DeclareMathOperator{\Aut}{Aut}
\DeclareMathOperator{\Bij}{{Bij}}
\DeclareMathOperator{\codim}{codim}
\DeclareMathOperator{\depth}{{depth}} 
\DeclareMathOperator{\Diff}{Diff}
\DeclareMathOperator{\Exc}{Exc}
\DeclareMathOperator{\GL}{{GL}}
\DeclareMathOperator{\Hom}{Hom}
\DeclareMathOperator{\id}{{id}}
\DeclareMathOperator{\im}{{im}}
\DeclareMathOperator{\PGL}{{PGL}}
\DeclareMathOperator{\Spec}{{Spec}}
\DeclareMathOperator{\Supp}{{Supp}}

\DeclareMathAlphabet{\mathchanc}{OT1}{pzc}%
                                 {m}{it}
\newcommand{\mcH}{\mathchanc{H}}
\newcommand{\mco}{\mathchanc{o}}
\newcommand{\mcm}{\mathchanc{m}}
\newcommand{\sHom}[0]{{\mcH\mco\mcm}}

\newcommand{\fm}{\mathfrak{m}}
\newcommand{\fn}{\mathfrak{n}}

\DeclareMathOperator{\Char}{char}

\DeclareMathOperator{\pr}{pr}

\usepackage[shortlabels]{enumitem}
\setlist[enumerate,1]{label={(\alph*)}, ref={\alph*}}

\EpigaVolumeYear{7}{2023} \EpigaArticleNr{5} \ReceivedOn{December 6, 2021}
\InFinalFormOn{August 23, 2022}
\AcceptedOn{October 23, 2022}

\title{Abundance for slc surfaces over arbitrary fields}
\titlemark{Abundance for slc surfaces over arbitrary fields}

\author{Quentin Posva}
\address{\'{E}cole Polytechnique F\'{e}d\'{e}rale de Lausanne, SB MATH CAG, MA C3 595 (B\^atiment MA), Station 8, CH-1015 Lausanne, Switzerland}
\email{quentin.posva@epfl.ch}

\authormark{Q.~Posva}

\AbstractInEnglish{We prove the abundance conjecture for projective slc surfaces over arbitrary fields of positive characteristic. The proof relies on abundance for lc surfaces over arbitrary fields, proved by H.~Tanaka, and on the technique of C.~Hacon and C.~Xu to descend semi-ampleness from the normalization. As a corollary, we obtain the existence of good minimal models for lc threefolds of general type over $F$-finite fields.}



\acknowledgement{ Financial support was provided by the European Research Council under grant number 804334.}

\begin{document}


\maketitle

\begin{prelims}

\DisplayAbstractInEnglish

\bigskip

\DisplayKeyWords

\medskip

\DisplayMSCclass

\end{prelims}


\newpage

\setcounter{tocdepth}{1}

\tableofcontents


\section{Introduction}

The minimal model program (MMP) predicts that a variety with mild singularities $X$ admits a birational model $X'$ such that either $K_{X'}$ is nef, or  there is a fibration $X'\to Y$ whose general fiber is a Fano variety. In the first case, the MMP is completed by the abundance conjecture: if $K_{X'}$ is nef, then it should also be semi-ample. 

In the case of surfaces, both the MMP and the abundance conjecture are established in many cases. For smooth surfaces over the complex numbers, this goes back to the work of the Italian school at the beginning of the twentieth century and to the subsequent work of Kodaira, although the results were formulated in different terms; see~\cite[Section~1]{Matsuki_Introduction_to_the_Mori_Program} for an exposition of these results. These classical methods were extended by Mumford~\cite{Mumford_Enriques_classification_in_char_p_I} to surfaces over algebraically closed fields of positive characteristic. Since then, the MMP and the abundance conjecture were proved more generally for log canonical surface pairs over the complex numbers by Fujino~\cite{Fujino_MMP_for_log_surfaces} and over algebraically closed fields of positive characteristic by Tanaka~\cite{Tanaka_MMP_and_abundance_for_positive_characteristic_log_surfaces}.

The work of Koll\'{a}r and Shepherd-Barron on the moduli space of canonically polarized smooth complex surfaces~\cite{Kollar_Shepherd_Barron_3folds_and_deformations_of_surfaces_singularities} has demonstrated that in order to have a good moduli theory of such surfaces, we should consider the larger class of so-called semi-log canonical (slc) surfaces. Hence it is natural to ask whether the MMP and the abundance theorem can be extended to that class of surfaces. As a matter of fact, the usual MMP does not work; see~\cite[Example 5.4]{Fujino_Fundamental_Theorems_for_slc_pairs}. On the other hand, abundance holds for slc surface pairs in characteristic zero by~\cite[Sections~8 and~12]{Flips_and_abundance_for_3folds} and over algebraically closed fields of positive characteristic by~\cite{Tanaka_Abundance_for_slc_surfaces}.

\bigskip
The purpose of this article is to extend the abundance theorem to slc surfaces over any field of positive characteristic. We prove the following. 

\begin{theorem_intro}\label{theorem:abundance_relative}
Let $(S,\Delta)$ be an slc surface pair and $f\colon S\to B$ a projective morphism, where $B$ is quasi-projective over a field of positive characteristic. Assume that $K_{S}+\Delta$ is $f\!$-nef; then it is $f\!$-semi-ample.
\end{theorem_intro}

Let us sketch the proof in the case where $B$ is the spectrum of a field. Abundance holds over arbitrary fields for lc surface pairs by the work of Tanaka~\cite{Tanaka_Abundance_for_lc_surfaces_over_imperfect_fields}. Thus if $(\bar{S},\bar{D}+\bar{\Delta})$ is the normalization of $(S,\Delta)$, since $K_S+\Delta$ pull backs to $K_{\bar{S}}+\bar{D}+\bar{\Delta}$, the latter is semi-ample. We have to find a way to descend semi-ampleness along the normalization.

Our strategy is similar to that of~\cite{Hacon_Xu_Finiteness_of_B_representations_and_slc_abundance}. Let $\tau$ be the involution of $\bar{D}^n$ induced by the normalization, and let $\varphi\colon \bar{S}\to \bar{T}$ be the fibration given by a sufficiently divisible multiple of $K_{\bar{S}}+\bar{D}+\bar{\Delta}$. One shows that the set-theoretic equivalence relation on $\bar{T}$ induced by $(\varphi,\varphi\circ\tau)\colon \bar{D}^n\rightrightarrows \bar{T}$ is finite using finiteness of \textbf{B}-representations. It follows that the quotient $T:=\bar{T}/(\bar{D}^n\rightrightarrows \bar{T})$ exists, and similar arguments show that the hyperplane divisor of $\bar{T}$ descends to $T$. Then it is not difficult to show that the composition $\bar{S}\to\bar{T}\to T$ factors through $S$ and that a multiple of $K_S+\Delta$ is the pullback of the hyperplane divisor of $T$.

In~\cite{Hacon_Xu_Finiteness_of_B_representations_and_slc_abundance} the authors use the theory of sources and springs of a crepant log structure developed by Koll\'{a}r (see~\cite[Section~4.3]{Kollar_Singularities_of_the_minimal_model_program}) to prove the finiteness of the equivalence relation and descent of the hyperplane divisor. We have not placed our proof on such axiomatic ground since the few cases of crepant log structures $\bar{S}\to\bar{T}$ that arise in our situation can be quite explicitly described. However, our proof is an illustration of Koll\'{a}r's theory: the technical details are easier, yet we encounter its main steps and subtleties. We highlight the correspondences in \Cref{remark:Kollars_theory_illustrated}. 

There is another approach to slc abundance, developed in characteristic zero by Fujino~\cite{Fujino_Abundance_for_slc_3folds_announce, Fujino_Abundance_for_slc_3folds} and used by Tanaka to prove slc abundance for surfaces over algebraically closed fields of positive characteristic; see~\cite{Tanaka_Abundance_for_slc_surfaces}. This approach is actually closely related to that of Hacon and Xu, and to Koll\'{a}r's theory of crepant log structures: the finiteness of \textbf{B}-representations plays a crucial role (see~\cite[Conjecture 4.2]{Fujino_Abundance_for_slc_3folds_announce}), and the geometric properties of $\bar{S}\to\bar{T}$ that are relevant in Fujino's approach (see~\cite[Proposition 3.1]{Fujino_Abundance_for_slc_3folds_announce}) can be understood in terms of sources, springs and $\bP^1$-links (see~\cite[Section~4.3]{Kollar_Singularities_of_the_minimal_model_program}).

The setup of~\cite{Hacon_Xu_Finiteness_of_B_representations_and_slc_abundance} may also be applied to the relative setting $f\colon S\to B$. However, instead of adapting all the previous steps to this relative setting, we choose to reduce to the absolute case as in~\cite{Tanaka_Abundance_for_lc_surfaces_over_imperfect_fields}. The main step is to compactify both $S$ and $B$ while preserving the properties of $S$ and $f$. This is achieved by running an MMP and carrying the gluing data along it.

\subsection*{Applications}
We give two applications of \Cref{theorem:abundance_relative}. The first one is about families of slc surfaces in mixed characteristic. In positive characteristic, relative semi-ampleness is a property of fibers by~\cite{Cascini_Tanaka_Relative_semi_ampleness_in_pos_char}. Recent work of Witaszek~\cite{Witaszek_Relative_semiampleness_in_mixed_char} shows that a similar statement holds for relative semi-ampleness in mixed characteristic. Combining these results with our main theorem, we obtain the following. 

\begin{theorem_intro}[\Cref{theorem:abundance_in_mixed_characteristic}]
Let $S$ be an excellent regular one-dimensional scheme that is separated and of finite type over $\Spec(\bZ)$. Assume that no component of $S$ is of equicharacteristic zero. Let $f\colon (X,\Delta)\to S$ be a surjective flat projective morphism of relative dimension two with $S_2$ fibers. Let $\Delta$ be a $\bQ$-divisor on $X$ such that $(X,\Delta+X_s)$ is slc for every closed point $s\in S$.

If $K_X+\Delta$ is $f\!$-nef, then it is $f\!$-semi-ample.
\end{theorem_intro}

The second application is abundance for minimal lc threefold pairs of general type over $F$-finite fields of big enough characteristic. 

\begin{theorem_intro}[\Cref{thm:abundance_in_dim_3}]
Let $(X,\Delta)$ be a three-dimensional lc projective pair over an $F$-finite field $k$ of characteristic $p>5$. Let $f\colon X\to S$ be a projective morphism onto a projective normal $k$-scheme satisfying $f_*\sO_X=\sO_S$. 

If\, $K_X+\Delta$ is $f\!$-nef and $f\!$-big, then it is $f\!$-semi-ample.
\end{theorem_intro}

Over an algebraically closed field of characteristic $p>5$, this result was obtained in \cite[Theorem~1.1]{Waldron_MMP_for_3_folds_in_char_>5}. The proof in \textit{op.\ cit.}\ relies on the MMP for klt pairs and on abundance for slc surfaces. The recent developments of the MMP over imperfect fields (see \cite{Das_Waldron_MMP_for_3folds_over_imperfect_fields} and also \cite{Bhatt&Co_MMP_for_3folds_in_mixed_char}) combined with our \Cref{theorem:abundance_relative} yield the $F$-finite case. The abundance for surfaces comes into play through the following form. 

\begin{theorem_intro}[\Cref{theorem:restriction_to_boundary_is_semi_ample}]
Let $(X,\Delta)$ be a projective $\bQ$-factorial dlt threefold over an arbitrary field $k$ of characteristic $p>5$. Assume that $K_X+\Delta$ is nef. Then $(K_X+\Delta)|_{\Delta^{=1}}$ is semi-ample.
\end{theorem_intro}

The proof we give follows closely that of \cite[Theorem~1.3]{Waldron_MMP_for_3_folds_in_char_>5}.

\subsection*{Acknowledgements}
I am grateful to Jakub Witaszek for several suggestions, and to the anonymous reviewer for their comments. I also thank the anonymous reviewer for some corrections and suggestions. 

\section{Preliminaries}

\subsection{Conventions and notation}
We work over an arbitrary field $k$ of positive characteristic, except in \Cref{section:mixed_characteristic}, where we work over an excellent base scheme $S$. We use the same terminology in the more general case, replacing $k$ with $S$ in the definitions below.

A \emph{variety} is a connected separated reduced equidimensional scheme of finite type over $k$. Note that a variety in our sense might be reducible. A \emph{curve} (resp.\ a \emph{surface}, a \emph{threefold}) is a variety of dimension one (resp.\ two,  three).

Let $X$ be an integral scheme. A coherent $\sO_X$-module $\sF$ is $S_2$ if it satisfies the condition $\depth_{\sO_{X,x}}\sF_x\geq \min\{2,\dim\sF_x\}$ for all $x\in X$.

If $X$ is a reduced Noetherian scheme, its \emph{normalization} is defined to be its relative normalization along the structural morphism $\bigsqcup_\eta\Spec(k(\eta))\to X$, where $\eta$ runs through the generic points of $X$. Recall that $X$ is normal if and only if it is regular in codimension one and $\sO_X$ is $S_2$.

Let $C$ be a regular irreducible proper curve over an arbitrary field $k$. We define the \textit{genus} of $C$ to be 
	$$g(C):=\frac{\dim_k H^1(C,\sO_C)}{\dim_k H^0(C,\sO_C)}=\dim_{H^0(C,\sO_C)}H^1(C,\sO_C).$$ 
Notice that $g(C)=1$ if and only if $\omega_{C/k}\cong \sO_C$; see~\cite[Example~7.3.35]{Liu_Ag_and_arithmetic_curves}.

Given a normal variety $X$, we denote by $K_X$ any Weil divisor on $X$ associated to the invertible sheaf $\omega_{X_\text{reg}}$. A $\bQ$-Weil divisor $D$ on $X$ is \emph{$\bQ$-Cartier} if for some $m>0$, the reflexive sheaf $\sO_X(mD)$ is invertible.

A \emph{pair} $(X,\Delta)$ is the data of a  normal variety $X$, together with a $\bQ$-Weil divisor $\Delta$ whose coefficients belongs to $[0,1]$, such that $K_X+\Delta$ is $\bQ$-Cartier. The divisor $\Delta$ is called the \emph{boundary} of the pair. 

We follow the standard terminology of~\cite[Section~2.1]{Kollar_Singularities_of_the_minimal_model_program} for the birational geometry of pairs. In particular, we refer the reader to \textit{loc.~cit.}\ for the notions of \emph{log canonical} (\emph{lc}\,) and \emph{divisorially log terminal} (\emph{dlt}\,) pairs, and for the definition of \emph{log canonical model}. 

If $(X,D)$ is a dlt pair, the \emph{strata} of $D^{=1}=\sum_{i=1}^nD_i$ are the irreducible components of all possible intersections $\bigcap_{i\in I}D_i$ with $I\subset \{1,\dots,n\}$.

Let $(X,\Delta+D)$ be a pair, where $D$ is a reduced divisor with normalization $D^n$. Then there is a ca\-non\-i\-cal\-ly defined $\bQ$-divisor $\Diff_{D^n}\Delta$ on $D^n$ such that restriction on $D^n$ induces an isomorphism ${\omega_X^{[m]}(m\Delta+mD)|_{D^n}}\cong \omega_{D^n}(m\Diff_{D^n}\Delta)$ for $m$ divisible enough. Singularities of $(X,\Delta+D)$ along $D$ and singularities of $(D^n,\Diff_{D^n}\Delta)$ are related by so-called adjunction theorems. We refer to~\cite[Section~4.1]{Kollar_Singularities_of_the_minimal_model_program} for fundamental theorems of adjunction theory.

Let $X$ be a variety and $C\subset X$ a curve proper over a field $k'$. For a Cartier divisor $L$ on $X$, we define the intersection number $L\cdot_{k'} C=\deg_{k'}\sO_C(L|_C)$. By linearity, this definition extends to $\bQ$-Cartier divisors.

Let $X$ be a variety and $D$ be a $\bQ$-divisor. We denote by $\Aut_k(X,D)$ the group of $k$-automorphisms $\sigma$ of $X$ with the property that $\sigma(D)=D$ (we allow $\sigma$ to permute the components of $D$ that have the same coefficients). Similarly, if $L$ is a line bundle, then $\Aut_k(X,L)$ is the group of $k$-automorphisms $\sigma$ of $X$ such that $\sigma^*L\cong L$.

\bigskip
The two main notions that appear in the abundance conjecture are those of relative nefness and semi-ampleness, which we recall now. Let $\pi\colon X\to S$ be a proper morphism of schemes and $L$ a line bundle on~$X$. 

We say that $L$ is \emph{relatively nef} over $S$ if, whenever $C\subset X$ is a proper curve contracted by $\pi$, it holds that $\deg L|_C\geq 0$. 

We say that $L$ is \emph{relatively semi-ample} over $S$ if, for $m\gg 1$, the natural map $\pi^*\pi_*L^{\otimes m}\to L^{\otimes m}$ is surjective. This notion is clearly local on the base $S$. If $S$ is the spectrum of a field $k$, then $L$ is semi-ample (over $k$) if and only if the linear system $H^0(X,L^{\otimes m})$ defines a morphism $X\to \bP_k(H^0(X,L^{\otimes m})^\vee)$ for $m\gg 1$.

Both notions holds for $L$ if and only if they hold for some power $L^{\otimes n}$ ($n\geq 1$). Thus they extend to the case where $L$ is a $\bQ$-Cartier divisor.

\subsection{Quotients by finite equivalence relations}\label{section:Quotients_by_equivalence_relation}
The theory of quotients by finite equivalence relations is developed in~\cite{Kollar_Quotients_by_finite_equivalence_relations} and~\cite[Section~9]{Kollar_Singularities_of_the_minimal_model_program}. For convenience, we recall the basic definitions and constructions that we will need.

Let $S$ be a base scheme and $X$ and $R$ two reduced $S$-schemes. An $S$-morphism $\sigma=(\sigma_1,\sigma_2)\colon R\to X\times_S X$ is a \emph{set-theoretic equivalence relation} if, for every geometric point $\Spec(K)\to S$, the induced map
		$$\sigma\left((\Spec(K)\right)\colon \Hom_S(\Spec(K),R)\to \Hom_S(\Spec(K),X)\times \Hom_S(\Spec(K),X)$$
	is injective and an equivalence relation of the set $\Hom_S(\Spec(K),X)$. We say in addition that $\sigma\colon R\to X\times_S X$ is \emph{finite} if both $\sigma_i\colon R\to X$ are finite morphisms.
	
	Suppose that $\sigma^\text{pre}\colon R\hookrightarrow X\times_S X$ is a reduced closed subscheme. Then there is a minimal set-theoretic equivalence relation generated by $\sigma^\text{pre}$; see~\cite[Section~9.3]{Kollar_Singularities_of_the_minimal_model_program}. Even if both $\sigma_i^\text{pre}\colon R^\text{pre}\to X$ are finite morphisms, the resulting pro-finite relation may not be finite: achieving transitivity can create infinite equivalence classes. 
	
	In most cases, we will be in the following situation. Let $Z\to X$ be a finite morphism and $\tau\colon Z\cong Z$ be an $S$-involution. The graph $\Gamma_\tau\subset Z\times_S Z$ is endowed with a finite morphism to $X\times_SX$; let $\sigma\colon \Gamma\hookrightarrow X\times_S X$ be its reduced image. Then we will be interested in the equivalence relation generated by $\sigma$, which we will denote by $R(\tau)\rightrightarrows X$.

	Let $(\sigma_1,\sigma_2)\colon R\to X\times_S X$ be a finite set-theoretic equivalence relation. A \emph{geometric quotient} of this relation is an $S$-morphism $q\colon X\to Y$ such that
		\begin{enumerate}
			\item $q\circ \sigma_1=q\circ \sigma_2$,
			\item $(Y,q\colon X\to Y)$ is initial in the category of algebraic spaces for the property above, and
			\item $q$ is finite.
		\end{enumerate}
The most important result for us is that quotients by finite equivalence relations usually exist in positive characteristic: if $X$ is essentially of finite type over a field $k$ of positive characteristic and $R\rightrightarrows X$ is a finite set-theoretic equivalence relation, then the geometric quotient $X/R$ exists and is a $k$-scheme; see~\cite[Theorem 6 and~Corollary 48]{Kollar_Quotients_by_finite_equivalence_relations}.

\begin{remark}
It is standard to add an extra condition in the definition of a geometric quotient $q\colon X\to Y$ that the geometric fibers are the $R$-equivalence classes: more precisely that, for every geometric point $\Spec(K)\to S$, the fibers of $q_K\colon X_K(K)\to Y_K(K)$ are the $\sigma(R_K(K))$-equivalence classes of $X_K(K)$. But this condition turns out to be a consequence of the three other ones; see the proof of~\cite[Lemma 17]{Kollar_Quotients_by_finite_equivalence_relations}.
\end{remark}

We also record the following result.

\begin{lemma}\label{lemma:descent_of_line_bdl_along_quotient}
Let $X$ be a reduced Noetherian pure-dimensional scheme and $R\rightrightarrows X$ a finite equivalence relation for which there exists a finite quotient $q\colon X\to Y:=X/R$. Let $L$ be a line bundle on $Y$, with pullback $L_X=q^*L$. Then $L$ is equal to the subsheaf of $q_*L_X$ formed by those sections which are $R$-invariant.
\end{lemma}

{\samepage\begin{proof}
Denote by $\sigma_1,\sigma_2\colon R\rightrightarrows X$ the two projection morphisms. By \cite[Lemma~9.10]{Kollar_Singularities_of_the_minimal_model_program}, we have
		$$\sO_Y=\ker\left[ q_*\sO_X \overset{\sigma_1^*-\sigma_2^*}{\xrightarrow{\hspace{1.2cm}}} (q\circ\sigma_i)_*\sO_{R}\right].$$
Tensor this expression by $L$ and use the projection formula to obtain the result.
\end{proof}
}
  
\subsection{Demi-normal varieties and slc pairs}
We recall the definitions of demi-normal varieties and of slc pairs. We refer the reader to~\cite[Section~3]{Posva_Gluing_for_surfaces_and_threefolds} for basic properties worked out in the generality we need.

\begin{definition}
A one-dimensional reduced Noetherian local ring $(R,\fm)$ is called a \emph{node} if there exists a ring isomorphism $R\cong S/(f)$, where $(S,\fn)$ is a regular two-dimensional local ring and $f\in \fn^2$ is an element that is not a square in $\fn^2/\fn^3$.

A locally Noetherian reduced scheme (or ring) is called \emph{nodal} if its codimension one local rings are regular or nodal. It is called \emph{demi-normal} if it is $S_2$ and nodal.
\end{definition}

Let $X$ be a reduced scheme with normalization $\pi\colon \bar{X}\to X$. The conductor ideal of the normalization is defined as
		$$\fI:=\sHom_X(\pi_*\sO_{\bar{X}},\sO_X).$$
It is an ideal in both $\sO_X$ and $\sO_{\bar{X}}$. We let
		$$D:=\Spec_X(\sO_X/\fI)\quad\text{and}\quad \bar{D}:=\Spec_{\bar{X}}(\sO_{\bar{X}}/\fI)$$
and call them the \emph{conductor subschemes}.

\begin{lemma}\label{lemma: basic properties of normalization of demi-normal scheme}
  Keep the notation above. Assume that $X$ is a demi-normal germ of variety.  Then: 
	\begin{enumerate}
		\item $D$ and $\bar{D}$ are reduced of pure codimension one.
		\item If $\eta\in D$ is a generic point such that $\Char k(\eta)\neq 2$, then the morphism $\bar{D}\to D$ is \'{e}tale of degree two in a neighborhood of $\eta$.
		\item $X$ has a dualizing sheaf which is invertible in codimension one. In particular, it has a well-defined canonical divisor $K_X$.
		\item Let $\Delta$ be a $\bQ$-divisor on $X$ with no component supported on $D$ and such that $K_X+\Delta$ is $\bQ$-Cartier. If $\bar{\Delta}$ denotes the divisorial part of $\pi^{-1}(\Delta)$, then there is a canonical isomorphism
				$$\pi^*\omega_X^{[m]}(m\Delta)\cong \omega_{\bar{X}}^{[m]}\left(m\bar{D}+m\bar{\Delta}\right)$$
		for $m$ divisible enough.
	\end{enumerate}
\end{lemma}

\begin{proof}
  See~\cite[Lemma 2.3.3]{Posva_Gluing_for_surfaces_and_threefolds}.
\end{proof}

\begin{definition}
We say that $(X,\Delta)$ is a \emph{semi-log canonical $($slc\,$)$ pair} if $X$ is demi-normal, $\Delta$ is a $\bQ$-divisor with no components along $D$, $K_X+\Delta$ is $\bQ$-Cartier, and the normalization $(\bar{X}, \bar{D}+\bar{\Delta})$ is an lc pair.
\end{definition}

\begin{definition}
We keep the notation of \Cref{lemma: basic properties of normalization of demi-normal scheme}. Let $\eta\in D$ be a generic point. Then $\bar{D}\to D$ is either purely inseparable (that can only happen if $\Char k(\eta)=2$) or generically \'{e}tale. In the first case, we call $(\eta\in X)$ an \emph{inseparable node}; in the second case, we call it a \emph{separable node}. 
\end{definition}

\begin{lemma}\label{lemma:normalization_gives_involution}
Keep the notation above. Assume that $X$ has only separable nodes.  Then: 
	\begin{enumerate}
		\item The induced morphism of normalizations $\bar{D}^n\to D^n$ is the geometric quotient by a Galois involution $\tau$.
		\item If $K_X+\Delta$ is $\bQ$-Cartier, then $\tau$ is a log involution of $(\bar{D}^n,\Diff_{\bar{D}^n}\bar{\Delta})$.
	\end{enumerate}
\end{lemma}

\begin{proof}
  See~\cite[Lemma 2.3.5]{Posva_Gluing_for_surfaces_and_threefolds}.
\end{proof}

\begin{examples}
We give two examples of demi-normal surfaces to illustrate the previous definitions.
	\begin{enumerate}
		\item Let $k$ be any field and $S=V(x_0x_1x_2)=P_0\cup P_1\cup P_2\subset\bA^3$ be the union of the three coordinate planes in the affine three-dimensional $k$-space. It is a demi-normal surface with only separable nodes. Its conductor subscheme $D\subset S$ is the union of the three coordinate lines. Its normalization $\bar{S}=P_0\sqcup P_1\sqcup P_2$ is the disjoint union of these coordinate planes. Let $y_i$ and $z_i$ be the coordinates on~$P_i$. Then the conductor subscheme $\bar{D}\subset \bar{S}$ is the union of the coordinate axes:
		$$\bar{D}=\bigsqcup_{i=0}^2 L(y_i)\cup L(z_i),$$ 
where $L(y_i)$ is the axis corresponding to $y_i$. Notice that $\bar{D}$ is not normal at the origins: its normalization $\bar{D}^n$ is the disjoint union of the axes. The involution $\tau$ on $\bar{D}^n$ is given by the identifications $L(y_i)\cong L(z_{i+1})$, $y_i\mapsto z_{i+1}$, where the index is taken modulo $3$.
		\item Let $k$ be a field of characteristic two, and consider the surface $S=V(wv^2-u^2)$. It is a demi-normal surface with a single node, along the conductor subscheme $D$ defined by $u=v=0$. Its normalization is $\bA^2$ with coordinates $t={u}/{v}$ and $v$. The conductor $\bar{D}\subset \bA^2_{t,v}$ is given by $v=0$. Thus the morphism $\bar{D}\to D$ is the morphism
				$$\bA^1_{t}\to \bA^1_w,\quad t\mapsto w=t^2,$$
		which is the Frobenius morphism. In particular, $S$ has an inseparable node.
	\end{enumerate}
\end{examples}

\begin{proposition}\label{prop:structure_of_normalization}
Let $X$ be a demi-normal variety over a field $k$ of positive characteristic and $\pi\colon \bar{X}\to X$ its normalization. Then we have a factorization
		$$\pi=\left( \bar{X}\overset{\nu}{\longrightarrow}\tilde{X}\overset{F}{\longrightarrow} X\right),$$
where
	\begin{enumerate}
		\item $\nu$ is the geometric quotient of\, $\bar{X}$ by the finite set-theoretic equivalence relation induced by the separable nodes of\, $X$;
		\item $F$ is a purely inseparable morphism induced by the inseparable nodes of\, $X$ $($see~\cite[Section~3.5]{Posva_Gluing_for_surfaces_and_threefolds}$)$, and $F=\id$ if $\Char k\neq 2$.
	\end{enumerate}
\end{proposition}

\begin{proof}
  \cite[Proposition 3.1.12]{Posva_Gluing_for_surfaces_and_threefolds}.
\end{proof}

If a scheme fails to satisfy the property $S_2$, in many cases it has a finite alteration that is $S_2$. 

\begin{proposition}\label{proposition:S_2_fication}
Let $X$ be a reduced equidimensional excellent scheme. Then the locus $U$ where $X$ is $S_2$ is an open subset with $\codim_X(X\setminus U)\geq 2$, and there exists a morphism $g\colon X'\to X$ such that
	\begin{enumerate}
		\item $X'$ is $S_2$ and reduced,
		\item $g$ is finite and an isomorphism precisely above $U$, and
		\item the normalization $X^n\to X$ factorizes through $g$.
	\end{enumerate}
We call $g\colon X'\to X$ the $\mathbf{S_2}$\emph{-fication} of\, $X$.
\end{proposition}

\begin{proof}
The morphism $g\colon X'\to X$ is the one given by~\cite[Proposition~5.10.16]{EGA_IV.2} (see~\cite[Section~5.10.13]{EGA_IV.2} for the definition of the $Z^{(2)}$ appearing there). The first two items also follow from~\cite[Proposition~5.10.16]{EGA_IV.2} granted that $g$ is finite, which holds by~\cite[Proposition~5.11.1]{EGA_IV.2}. The fact that $g$ factors the normalization follows from the finiteness of $g$ and from~\cite[035Q]{Stacks_Project}.
\end{proof}

\subsection{Preliminary results}

The first four results below are probably well known; we provide proofs for convenience.

\begin{proposition}\label{proposition:log_auto_of_curves}
Let $C$ be a regular projective curve over a field $k$ and $D$ a boundary such that $K_C+D$ is ample. Then $\Aut_k(C,D)$ is finite.
\end{proposition}

\begin{proof}
We may replace $D$ with $\lceil D\rceil$. Since $K_C+D$ is ample and preserved by the elements of $\Aut_k(C,D)$, we can describe the latter group as the $k$-points of the linear algebraic group
		$$G:=\left\{\Phi\in\PGL_k H^0(C,\sO(m(K_C+D)))\mid \Phi(D)=D\right\}, \quad m\text{ divisible enough.}$$
	The tangent space of $G$ at the identity morphism is given by  $H^0(C,T_C\otimes I_D)$, see~\cite[Section~2.9]{Debarre_Higher_dimensional_algebraic_geometry}, which is trivial since $T_C\otimes I_D=\sO(-K_C-D)$ is anti-ample. It follows that $G$ is a finite group scheme, and thus $\Aut_k(C,D)$ is a finite group.
\end{proof}

\begin{lemma}\label{lemma:reduced_boundary_dlt_surface_pair}
Let $(S,\Delta)$ be a dlt surface over an arbitrary field. Then $S$ is $\bQ$-factorial, and the irreducible components of\, $\lfloor \Delta\rfloor$ are normal.
\end{lemma}

\begin{proof}
The $\bQ$-factorial property is proved in~\cite[Corollary~4.11]{Tanaka_MMP_for_excellent_surfaces}. Hence to show that the components of $\lfloor\Delta\rfloor$ are normal, we may assume that $\Delta=\lfloor \Delta\rfloor$ is irreducible. Then we can repeat the proof of~\cite[Proposition~5.51]{Kollar_Mori_Birational_geometry_of_algebraic_varieties}, using~\cite[Proposition~3.2]{Tanaka_MMP_for_excellent_surfaces} instead of~\cite[Corollary~2.68]{Kollar_Mori_Birational_geometry_of_algebraic_varieties}.
\end{proof}

\begin{lemma}\label{lemma:zero_dim_strata_on_fibers}
Let $(S,\Delta)$ be a dlt surface over an arbitrary field such that $K_S+\Delta$ is big and semi-ample, with associated birational proper morphism $\varphi\colon S\to T$. Let $t\in T$ and $\Theta$ be the $($possibly empty\,$)$ maximal sub-curve of $\Delta^{=1}$ contained in $\varphi^{-1}(t)$. Then the minimal strata of\, $\Theta$ are all isomorphic to each other.
\end{lemma}

\begin{proof}
We may of course assume that $\Theta$ is non-empty. By~\cite[Proposition~2.36]{Kollar_Singularities_of_the_minimal_model_program}, the curve $\Theta=\sum_{i=1}^n\Theta_i$ is connected, and it has regular irreducible components by \Cref{lemma:reduced_boundary_dlt_surface_pair}. If $n=1$, the unique stratum is $\Theta$, and there is nothing to show. So we assume that $n\geq 2$: by connectedness, the minimal strata are the zero-dimensional ones. Write $\Diff_{\Theta_i}(\Delta-\Theta_i)=\Delta_{\Theta_i}$, so that we have $(K_S+\Delta)|_{\Theta_i}=K_{\Theta_i}+\Delta_{\Theta_i}\sim 0$. We also let $k_i=H^0(\Theta_i,\sO_{\Theta_i})$.

Fix an index $i$, and let $J(i)= \{1\leq j\leq n\mid j\neq i \text{ and } \Theta_i\cap\Theta_j\neq \emptyset\}$. Notice that $|J(i)|>0$ by connectedness and our current assumption $n\geq 2$. By~\cite[Theorem~2.31]{Kollar_Singularities_of_the_minimal_model_program}, the intersections $\Theta_i\cap \Theta_j=\{p_j\}$ are transversal, and $\sum_{j\in J(i)}p_j\leq \Delta_{\Theta_i}$. Since $K_{\Theta_i}\sim -\Delta_{\Theta_i}$, we see that $K_{\Theta_i}$ has negative degree. Thus by~\cite[Section~10.6]{Kollar_Singularities_of_the_minimal_model_program} we have $\deg_{k_i}K_{\Theta_i}=-2$. As $\deg_{k_i}p_j\geq 1$, we deduce that $|J(i)|\leq 2$, with equality if and only if $k(p_{j_1})=k_i=k(p_{j_2})$.

It is now easy to show, for example by induction on $n\geq 2$, that if $\Theta$ has more than one zero-dimensional stratum, then each one of them is contained in a $\Theta_i$ with $|J(i)|=2$. By connectedness, we obtain that any two zero-dimensional strata are isomorphic.
\end{proof}

\begin{proposition}[Castelnuovo's contraction criterion]\label{proposition:Castelnuovo_criterion}
Let $S\to B$ be a projective morphism from a surface to a Noetherian scheme that is quasi-projective over a field or a DVR. Let $C\subset S$ be a proper integral curve that is vertical over $B$. Assume that $S$ is regular in a neighborhood of\, $C$, that $K_S\cdot C<0$ and $C^2<0$. Then:   
	\begin{enumerate}
		\item $C$ is regular. 
		\item If\, $k=H^0(C,\sO_C)$, then $K_S\cdot_k C=-1=C\cdot_k C$.
		\item There exists a birational proper $B$-morphism $\varphi\colon S\to S'$ to a projective $B$-surface, that contracts $C$ to a regular point, and such that $\varphi\colon S\setminus C\to S'\setminus \varphi(C)$ is an isomorphism.
	\end{enumerate}
\end{proposition}

\begin{proof}
First we show that $C$ is regular. Let $C^\nu$ be its normalization with $k'=H^0(C^\nu,\sO_{C^\nu})$. Then by adjunction, we have
		$$K_S+C|_{C^\nu}\sim K_{C^\nu}+\Diff(0).$$
The divisor $\Diff(0)$ is effective, and since $K_S+C$ is Cartier (in a neighborhood of $C$), we obtain that $\Diff(0)$ is a $\bZ$-divisor; see~\cite[Proposition~2.35]{Kollar_Singularities_of_the_minimal_model_program}. By assumption, the degree of $K_S+C|_C$ is negative, and therefore the degree of $K_S+C|_{C^\nu}$ is also negative; see~\cite[Corollary~7.5.8]{Liu_Ag_and_arithmetic_curves}. By~\cite[Section~10.6]{Kollar_Singularities_of_the_minimal_model_program}, it follows that $\deg_{k'}K_{C^\nu}=-2$. Therefore, $\Diff(0)$ is either empty or a single point with coefficient one. In both cases, $(C^\nu,\Diff(0))$ is an lc curve, and by inversion of adjunction for surfaces, see~\cite[Theorem~5.1]{Tanaka_MMP_for_excellent_surfaces}, we obtain that $(S,C)$ is lc in a neighborhood of $C$. If $C$ is not regular, we see by~\cite[Theorem~2.31]{Kollar_Singularities_of_the_minimal_model_program} that it has a node; but then $\Diff(0)$ would be the sum of two distinct points with coefficient one. Thus we see that $C$ is regular.

Let $k=k'=H^0(C,\sO_C)$. We denote by $\chi_k$ the Euler characteristic, where the dimensions are taken over $k$. Since $C$ is regular, by the adjunction formula we have
		\begin{eqnarray}\label{eqn:adjunction}
		-2\chi_k(\sO_C)=\deg_k K_C=(K_S+C)\cdot_k C<0.
		\end{eqnarray}
We also claim that $\deg_k C^2=-h^0(C,\sO_C)$. We have $0>K_S\cdot_k C=\deg_k K_C-C^2$, which implies that $\deg_k K_C<0$ and thus $h^1(C,\sO_C)=0$. By~\cite[Section~10.6]{Kollar_Singularities_of_the_minimal_model_program}, we obtain that $\deg_k K_C=-2$. Looking again at \Cref{eqn:adjunction}, we find that $h^0(C,\sO_C)=-1=\deg_k C^2$, as claimed. 
		
By a generalization of Castelnuovo's criterion, see~\cite[Theorem~27.1]{Lipman_Rational_sing_and_applications_to_surfaces}, the facts that $\chi(\sO_C)>0$ and $-h^0(\sO_C)=C^2$ imply that there exists a proper morphism $\varphi\colon S\to S'$ as in the third statement.
\end{proof}

The next two results study the pluricanonical representations on regular curves of genus zero.

\begin{lemma}\label{lemma:change_of_fields_and_representations}
Let $X$ be a proper variety over an arbitrary field $k$ and $L$ a line bundle on $X$. Write $K:=H^0(X,\sO_X)$. If the natural representation $\rho_K\colon\Aut_K(X,L)\to \GL_K H^0(X,L)$ has finite image, then so does $\rho_k\colon \Aut_k(X,L)\to \GL_k H^0(X,L)$.
\end{lemma}

\begin{proof}
If $\varphi\in \Aut_k(X)$, then $\varphi^*\colon K\to K$ is a $k$-linear field automorphism. This gives a partition
		$$\Aut_k(X,L)=\bigsqcup_{\sigma\in\Aut_k(K)}\Aut_k^\sigma(X,L)$$
which is finite since $k\subset K$ is a finite field extension. Notice that if $\varphi\in \Aut_k^\sigma(X,L)$ and $\psi\in \Aut_k^\varsigma(X,L)$, then $\varphi\circ \psi\in \Aut_k^{\sigma\circ \varsigma}(X,L)$. Thus $\Aut_K(X,L)=\Aut_k^{\id}(X,L)$ is a subgroup of finite index of $\Aut_k(X,L)$.

The map $\rho_k\colon \Aut_k(X,L)\to \GL_k H^0(X,L)$ is a group morphism, and therefore it sends cosets to cosets. Thus $\im(\rho_k)$ is the union of finitely many cosets of $\im(\rho_K)=\rho_k(\Aut_k^{\id}(X,L))$. Since $\im(\rho_K)$ is finite by assumption, we obtain that $\im(\rho_k)$ is finite.
\end{proof}

\begin{proposition}\label{prop:pluricanonical_rep_for_genus_zero}
Let $C$ be a regular proper curve of genus zero over an arbitrary field $k$ and $E$ an effective $\bQ$-divisor such that $K_C+E\sim_{\bQ}0$. Then for $m$ divisible enough, the natural representation $\Aut_k(C,E)\to\GL_k H^0(C,\omega_C^m(mE))$ has finite image.
\end{proposition}

\begin{proof}
By \Cref{lemma:change_of_fields_and_representations}, we may replace $k$ by $H^0(C,\sO_C)$ to prove the result. If $C$ is smooth over $k$, then we may assume that $k$ is algebraically closed. Then $C\cong \bP^1_k$, and $\Supp(E)$ contains at least two points. If it actually contains more than three points, then $\Aut_k(\bP^1,E)$ is finite. If it contains exactly two points, we may assume $\Supp E=\{0,\infty\}\subset\bP^1$. If $x$ denotes the coordinate on $\bP^1$, then $\Aut_k(\bP^1,E)\cong k^*$ acts on $H^0(\bP^1,\omega_{\bP^1}(E))=k\cdot \frac{dx}{x}$ by scaling of $x$, and this action is trivial.

For the rest of the proof, we assume that $C$ is non-smooth over $k$. By~\cite[Proposition~9.8 and Theorem~9.10]{Tanaka_Invariants_of_varieties_over_imperfect_fields}, it holds that $\Char k=2$, and we can find degree two purely inseparable extensions $k\subset l\subset k'$ such that $C_l=C\otimes_kl$ is integral with non-isomorphic normalization $B\cong \bP^1_{k'}$:
	$$\begin{tikzcd}
	B\arrow[d, "\nu"]\arrow[drr, "g"] && \\
	C_l\arrow[rr, "f" below] && C\rlap{.}
	\end{tikzcd}$$
Moreover, there exists a $k'$-rational point $P\in B(k')$ such that $K_B+P\sim g^*K_C$. Thus if $m>0$ is such that $mE$ is a $\bZ$-divisor, then $mK_B+mP+g^{*}(mE)\sim g^*(mK_C+mE)\sim 0$. Since $C_l$ is reduced, we have an inclusion $\sO_{C_l}\subset \nu_*\sO_B$. Tensoring with $f^*\omega_C^m(mE)$ and using the projection formula, we obtain an inclusion $f^*\omega_C^m(mE)\subset \nu_*g^*\omega_C^m(mE)$. Taking global section, we get a sequence of inclusions
		\begin{equation}\label{eqn:inclusion_pullback_global_sections}
		H^0\left(C,\omega_C^m(mE)\right)\subset H^0\left(C,\omega_C^m(mE)\right)\otimes_k l\subset H^0\left(B, \omega_B^m(mP+mg^{*}(E))\right).
		\end{equation}
On the other hand, extending scalars along $k\subset l$ gives a natural map 
			$$g^*\colon \Aut_k(C,E)\to \Aut_{l}\left(B,P+g^{*}(E)\right),$$ 
	whose image respects the flag \eqref{eqn:inclusion_pullback_global_sections}. Therefore, it is sufficient to show that the representation of $\Aut_l(B,P+g^*(E))$ on $H^0(B,\omega_B^m(mP+mg^{*}(E)))$ is finite. By \Cref{lemma:change_of_fields_and_representations}, it is sufficient to prove that the representation of $\Aut_{k'}(B,P+g^*(E))$ is finite. By~\cite[Proposition~9.8]{Tanaka_Invariants_of_varieties_over_imperfect_fields}, we may choose the extensions $k\subset l\subset k'$ so that $P$ does not belong to the support of $g^{*}(E)$. Thus the support of $P+g^{*}(E)$ contains at least two points, and we may apply the argument of the first paragraph.
\end{proof}

The next proposition summarizes useful results contained in the proofs of~\cite[Proposition~5.2.9 and Lemma~5.2.12]{Posva_Gluing_for_surfaces_and_threefolds}.

\begin{proposition}\label{proposition:useful_results_from_gluing_article}
Let $\varphi\colon S\to T$ be a projective morphism from a surface $S$ to a variety $T\!$ of dimension at most one, with $\varphi_*\sO_S=\sO_T$, over an arbitrary field. Let $\Theta$ and $\Upsilon$ be divisors such that $(S,\Theta+\Upsilon)$ is dlt, $(\Theta+\Upsilon)^{=1}=\Theta$ and $K_S+\Theta+\Upsilon\sim_{\bQ,\varphi} 0$. Let $z\in T$ be a closed point and $W,W'$ be lc centers of $(S,\Theta+\Upsilon)$ that are minimal for the property that $\varphi(W)=\{z\}=\varphi(W')$. Then: 
	\begin{enumerate}
		\item There exists a log isomorphism $(W,\Diff^*(\Theta+\Upsilon))\cong (W',\Diff^*(\Theta+\Upsilon))$.
		\item If moreover $\dim T=1$ and $W$ is a genus one curve, then $W=W'$.
	\end{enumerate}
\end{proposition}

\begin{proof}
Since $(S,\Theta+\Upsilon)$ is dlt, its lc centers are the strata of $\Theta$, and we can perform adjunction in any codimension. In this situation, we write the different with $\Diff^*$ (see~\cite[Section~5.2.1]{Posva_Gluing_for_surfaces_and_threefolds}). Thus $W$ and $W'$ are among the strata of $\Theta$ that are contained in $\varphi^{-1}(z)$. 

To begin with, we prove the existence of the log isomorphisms. First assume that $\Theta\cap \varphi^{-1}(z)$ is connected. There is nothing to show if it is zero-dimensional. If it is one-dimensional, it is a chain of regular proper curves, and the minimal lc centers are the intersections of these curves. If there is more that one minimal lc center, then by adjunction it is easy to verify that they are $k(z)$-points.

We show that if $\Theta$ does not dominate $T$, then $\Theta\cap \varphi^{-1}(z)$ is connected (and so there exists a log isomorphism by the previous paragraph). We follow the method of the last part of the proof of~\cite[Lemma~5.2.12]{Posva_Gluing_for_surfaces_and_threefolds}, so we only sketch the argument. Notice that $\dim T=1$. We run a $(K_S+\Upsilon)$-MMP over $T$, which ends with a birational model~\cite{Tanaka_MMP_for_excellent_surfaces}:
		$$\begin{tikzcd}
		(S,\Theta+\Upsilon) \arrow[rr, "g"]\arrow[dr, "\varphi" below left] && (S',\Theta'+\Upsilon') \arrow[dl, "\varphi'"] \\
		& T\rlap{.} &
		\end{tikzcd}$$
Then $g$ is crepant, see~\cite[Remark~5.1.14]{Posva_Gluing_for_surfaces_and_threefolds}; $K_{S'}+\Upsilon'$ is $\varphi'$-nef; $(\Theta')^2\leq 0$, see~\cite[Lemma~5.1.5]{Posva_Gluing_for_surfaces_and_threefolds}; and we can write $\Theta'=(\varphi')^*N+E$ for $E>0$ vertical and a $\bQ$-divisor $N$ on $T$. Taking in account that $K_{S'}+\Upsilon'$ is numerically equivalent to the vertical divisor $-\Theta'$, we obtain
		$$0\geq \left(\Theta'\right)^2=(K_{S'}+\Upsilon')\cdot \left(\left(\varphi'\right)^*N+E\right)\geq 0,$$
	and thus $\Theta'$ is a reduced fiber; see~\cite[Lemma~5.1.5]{Posva_Gluing_for_surfaces_and_threefolds}. Therefore, $\Theta'\cap (\varphi')^{-1}(z)$ is connected, and since $g$ is crepant, it follows that $\Theta\cap \varphi^{-1}(z)$ is also connected; see~\cite[Lemma~5.1.13]{Posva_Gluing_for_surfaces_and_threefolds}.

To conclude the first point, we need to produce log isomorphisms in the case where $\Theta\cap \varphi^{-1}(z)$ is not connected and $\Theta$ does dominate $T$. We follow the method of~\cite[Proposition~5.2.9]{Posva_Gluing_for_surfaces_and_threefolds}; once again we only sketch the argument. Base-changing along an \'{e}tale morphism $(z'\in T')\to (z\in T)$ with $k(z)=k(z')$, we may assume that the different connected components of $\Theta\cap \varphi^{-1}(z)$ belong to different connected components of~$\Theta$ and that these components are all horizontal over $T$. We run a $(K_S+\Upsilon)$-MMP over $T$; it terminates with a Fano contraction, see~\cite{Tanaka_MMP_for_excellent_surfaces}:
		$$\begin{tikzcd}
		(S,\Theta+\Upsilon) \arrow[r, "f"]\arrow[d, "\varphi"] & (S'',\Theta''+\Upsilon'')\arrow[d, "p"] \\
		T & B\rlap{.}\arrow[l]
		\end{tikzcd}$$
Notice that $f$ is crepant; see~\cite[Remark~5.1.14]{Posva_Gluing_for_surfaces_and_threefolds}. By~\cite[Lemma~5.1.13]{Posva_Gluing_for_surfaces_and_threefolds}, the morphism $f$ induces a bijection between the connected components of $\Theta$ and $\Theta''$. There is a component $D_1\subset \Theta''$ that is $p$-ample. By assumption, there is another component $D_2\subset \Theta''$ that is disjoint from $D_1$. Take a curve $C$ that intersects $D_2$ and is contracted by $p$; then $D_2\cdot C<0$ cannot happen, for otherwise $D_1\cap D_2\neq \emptyset$. Thus $D_2$ is also $p$-ample. A similar argument shows that the fibers of $p$ are one-dimensional. Hence if $\eta$ is the generic point of $B$, then $S''_\eta$ is a regular proper curve with $k(\eta)=H^0(S''_\eta, \sO_{S''_\eta})$. Since 
		$$0=(K_{S''}+\Theta''+\Upsilon'')|_{S''_\eta},$$
we see that $\deg_{k(\eta)}K_{S''_\eta}=-2$, see~\cite[Section~10.6]{Kollar_Singularities_of_the_minimal_model_program}, and therefore $D_1|_{S''_\eta}$ and $D_2|_{S''_\eta}$ are $k(\eta)$-points. It follows that $D_1\to B$ and $D_2\to B$ are isomorphisms. We also deduce from the equality that, up to shrinking $T$ around $z$, we have $\Theta''=D_1+D_2$. Therefore, we obtain log isomorphisms on $S''$, and since $f$ is crepant, we get log isomorphisms on $S$ (see for example~\cite[Corollary~5.2.5]{Posva_Gluing_for_surfaces_and_threefolds}).

Let us now prove the second point. If $\Theta$ does not dominate $T$, then we have seen that $\Theta\cap \varphi^{-1}(z)$ is connected, and by adjunction it must be equal to $W$. Assume that $\Theta$ dominates $T$. We run a few steps of the $(K_S+\Upsilon)$-MMP over $T$, contracting the curves in $\varphi^{-1}(z)$ that do not belong to $\Theta$. Every such curve is eventually contracted since $K_S+\Upsilon\sim_{\bQ,\varphi}-\Theta$. Therefore, we reach a birational model where the transform of $W$ intersects the transform of $\Theta$. Since each step of the MMP is crepant for $(S,\Theta+\Upsilon)$, see~\cite[Remark~5.1.14]{Posva_Gluing_for_surfaces_and_threefolds}, it follows that $\Diff^*_W(\Theta+\Upsilon)\neq 0$, which contradicts adjunction. Thus $\Theta\cap \varphi^{-1}(z)$ actually does not contain a genus one curve.
\end{proof}

\section{Abundance for surfaces in positive characteristic}
\subsection{Absolute case}

We begin with abundance in the absolute case. 

\begin{theorem}\label{theorem:abundance_absolute_case}
Let $(S_0,\Delta_0)$ be a projective slc surface pair over an arbitrary field $k$ of positive characteristic. Assume that $K_{S_0}+\Delta_0$ is nef; then it is semi-ample.
\end{theorem}

For the duration of the proof, we fix $(S_0,\Delta_0)$ and let $(S,D+\Delta)$ be its normalization. Write $D=D_G+D_I$, where $D_G$ is the preimage of the separable nodes of $S_0$ and $D_I$ is the preimage of the inseparable ones. Let $\tau$ be the induced log involution of $(D_G^n,\Diff_{D_G^n}(\Delta+D_I))$.

We emphasize that $S_0$ is not assumed to be irreducible. Thus $(S,\Delta+D)=\bigsqcup_{i=1}^N (S_i,\Delta_i+D_i)$ is the disjoint union of its normal irreducible components.

\bigskip
We divide the proof in several steps.

\subsubsection{Reduction to separable nodes}\label{section:reduction_to_sep_nodes_in_abs_case}
This step is only necessary if $\Char k=2$. Applying \Cref{prop:structure_of_normalization}, we get a factorization 
		$$(S,D+\Delta)\longrightarrow (S',D_I'+\Delta')\overset{\mu}{\longrightarrow} (S_0,\Delta_0),$$
where $\mu$ is finite purely inseparable and $(S',D_I'+\Delta')$ is slc with only separable nodes. By~\cite[Lemma~2.11.(3)]{Cascini_Tanaka_Relative_semi_ampleness_in_pos_char},
if $K_{S'}+D_I'+\Delta'=\mu^*(K_{S_0}+\Delta_0)$ is semi-ample, then so is $K_{S_0}+\Delta_0$. Thus it suffices to study $(S',D_I'+\Delta')$, and so we may assume that $S_0$ has only separable nodes.
		
	
\subsubsection{Quotienting the fibration} The pullback $K_S+D+\Delta$ is nef by assumption, so it is semi-ample by~\cite{Tanaka_Abundance_for_lc_surfaces_over_imperfect_fields}. Choose $m>0$ even such that $m(K_S+D+\Delta)$ is base-point-free, and let $\varphi\colon S\to T$ be the corresponding fibration onto a normal projective variety. We let $H$ be the hyperplane Cartier divisor on $T$ with the property that $\varphi^*\sO(H)=\sO(m(K_S+D+\Delta))$.

Since $S_0$ has only separable nodes, it is the geometric quotient of its normalization $S$ by the finite equivalence relation induced by the involution $\tau$ on $(D^n,\Diff_{D^n}\Delta)$. This equivalence relation is generated by the two morphisms $(\iota,\iota'=\iota\circ \tau)\colon D^n\rightrightarrows S$. Let $\psi\colon D^n\to E$ be the fibration corresponding to the base-point-free divisor $m(K_S+D+\Delta)|_{D^n}$. Then we have a diagram
		\begin{equation}\label{eqn:Stein_factorization_of_quotient_fibr}
		\begin{tikzcd}
		D^n\arrow[rr, shift left, "\iota"]\arrow[rr, shift right, "\iota'" below]\arrow[d, "\psi"] && S\arrow[d, "\varphi"] \\
		E\arrow[rr, shift left, "j"]\arrow[rr, shift right, "j'" below] && T\rlap{,}
		\end{tikzcd}
		\end{equation}
where $(\psi,j)$ (resp.\ $(\psi,j')$) is the Stein factorization of $\varphi\circ \iota$ (resp.\ of $\varphi\circ \iota'$).	The two morphisms $(j,j')\colon E\rightrightarrows T$ are finite, and they generate a pro-finite equivalence relation on $T$. (We only care about the reduced image of this relation in $T\times_k T$; see the comment after~\cite[Section~9.1]{Kollar_Singularities_of_the_minimal_model_program}.)

We claim that this relation is actually finite and that $H$ descends to the quotient. We prove both claims below; for the moment assume they hold. Let $q\colon T\to T_0:=T/(E\rightrightarrows T)$ be the quotient. Since the compositions $q\circ \varphi\circ \iota$ and $q\circ\varphi\circ\iota'$ are equal, we obtain a morphism $\varphi_0\colon S_0\to T_0$ such that the diagram
		\begin{equation}\label{diagram:quotient_of_fibration}
		\begin{tikzcd}
		S\arrow[d, "n"]\arrow[r,"\varphi"] & T\arrow[d, "q"] \\
		S_0\arrow[r, "\varphi_0"] & T_0
		\end{tikzcd}
		\end{equation}
commutes. Moreover, there is an ample Cartier divisor $H_0$ on $T_0$ such that $q^*H_0 = H$. 

\begin{claim}
We have $\varphi_0^*\sO(H_0)= \sO(m(K_{S_0}+\Delta_0))$. In particular, $K_{S_0}+\Delta_0$ is semi-ample.
\end{claim}

\begin{proof}
Tensoring the inclusion $\sO_{S_0}\subset n_*\sO_S$ by $\varphi_0^*\sO(H_0)$ and using the projection formula, we obtain
	$$\varphi^*_0\sO_{T_0}(H_0)\subset n_*\sO_S(m(K_S+D+\Delta)).$$
By the commutativity of \Cref{diagram:quotient_of_fibration} and the definition of $\varphi$ and $q$, for $s\in\sO(H_0)$, we see that its pullback $\varphi_0^*s\in n_*\sO(m(K_S+D+\Delta))$ is a log pluricanonical section whose restriction to $D^n$ is $\tau$-invariant. Thus $\varphi^*_0\sO(H_0)\subseteq \sO(m(K_{S_0}+\Delta_0))$ by~\cite[Proposition 3.1.7]{Posva_Gluing_for_surfaces_and_threefolds}.
Conversely, by \Cref{lemma:descent_of_line_bdl_along_quotient}, a section $t$ of $q_*\sO(H)$ belongs to $\sO(H_0)$ if and only if $j^*t=(j')^*t$. Looking at the diagram \eqref{eqn:Stein_factorization_of_quotient_fibr}, we see that this condition is equivalent to $\iota^*\varphi^*t=(\iota')^*\varphi^*t$, which by~\cite[Proposition 3.1.7]{Posva_Gluing_for_surfaces_and_threefolds}
means that $\varphi^*t\in \sO(m(K_{S_0}+\Delta_0))$. We have obtained the inverse inclusion $\sO(m(K_{S_0}+\Delta_0))\subseteq \varphi^*_0\sO(H_0)$, which concludes the proof.
\end{proof}

\begin{remark}
In the diagram \eqref{diagram:quotient_of_fibration}, the Stein factorization of $S_0\to T_0$ need not be demi-normal.

For example, consider the product $T:=E\times \bP^1$ of an elliptic curve with a rational smooth curve over an algebraically closed field. Denote by $p_E$ and $p_{\bP^1}$ the projections onto the factors. Let $\Delta_T$ be the sum of one section of $p_{\bP^1}$ and three distinct sections of $p_E$. Then $K_T+\Delta_T$ is ample, and $(T,\Delta_T)$ is dlt. Let $\varphi\colon S\to T$ be the blow-up of two distinct points $p$ and $q$ that are zero-dimensional strata of $\Delta$, and let $E_p,E_q$ be the corresponding $\varphi$-exceptional divisors. Then $\varphi^*(K_T+\Delta_T)=K_S+\Delta_S+E_p+E_q$. So $\varphi\colon S\to T$ is the ample model of $(S,\Delta_S+E_p+E_q)$. Let $\tau\colon E_p\cong E_q$ be an isomorphism that sends $\Delta_S|_{E_p}$ to $\Delta_S|_{E_q}$. Then the quotient $r\colon S\to S_0:=S/R(\tau)$ exists, and $(S_0,\Delta_{S_0}+F)$ is slc with normalization $(S,\Delta_S+E_p+E_q)$, where $F=r(E_p)=r(E_q)$. On the other hand, the induced involution $E\rightrightarrows T$ is given by $p\sim q$, and so the fibration given by $|m(K_{S_0}+\Delta_{S_0}+F)|$ is $S_0\to T_0:=T/(p\sim q)$. However, $T_0$ is not demi-normal since it is not $S_2$.
\end{remark}

\subsubsection{Finiteness}
It remains to show finiteness and descent, and we begin by the former. It is convenient to reduce to the case where (each component) of $(S,\Delta+D)$ is dlt. 

\begin{claim}
In order to show that the equivalence relation induced by $E\rightrightarrows T$ is finite, we may assume that $(S,D+\Delta)$ is dlt.
\end{claim}

\begin{proof}
Indeed, let $\phi\colon (S_\text{dlt},D_\text{dlt}+\Delta_\text{dlt}+E)\to (S,D+\Delta)$ be a crepant dlt blow-up, where $D_\text{dlt}=\phi^{-1}_*D$ and $E=\Exc(\phi)$; see~\cite[Theorem~4.7 and Remark~4.8]{Tanaka_MMP_for_excellent_surfaces}. Then $K_{S_\text{dlt}}+D_\text{dlt}+\Delta_\text{dlt}+E$ is semi-ample, and the corresponding fibration is just $\varphi\circ\phi\colon S_\text{dlt}\to S\to T$. Moreover, $(D^n, \Diff_{D^n}\Delta)=(D_\text{dlt}^n,\Diff_{D_\text{dlt}^n}(\Delta_\text{dlt}+E))$, so we recover the involution $\tau$ on the dlt model. Notice that $D_\text{dlt}^n$ is just the disjoint union of its irreducible components by \Cref{lemma:reduced_boundary_dlt_surface_pair}.
\end{proof}

Now let us write $T=\bigsqcup_{i\geq 1} T_i$ and $\varphi=\prod_i\varphi_i$, where $\varphi_i\colon S_i\to T_i$ is the fibration given by $m(K_{S_i}+\Delta_i+D_i)$. Let $\kappa_i:=\kappa(S_i,\Delta_i+D_i)\geq 0$ be the respective Kodaira dimensions; it holds that $\dim T_i=\kappa_i$. We denote by $\pi_0(D^n)$ the collection of irreducible components of $D^n$.
For $\Gamma\in \pi_0(D^n)$ we write $\Delta_\Gamma:=\Diff_\Gamma(\Delta+D-\Gamma)$. 

The following claim follows immediately from the construction.

\begin{claim}\label{claim:description_of_involution_downstairs}
The two morphisms $(j,j')\colon E\rightrightarrows T$ come from an involution $B^\tau$ on $E$ defined as follows: if $\tau(\Gamma)=\Gamma'$, then $B^\tau\colon \psi(\Gamma)\cong \psi(\Gamma')$ is induced by the isomorphism 
	$\tau^*\colon H^0(\Gamma', m(K_{\Gamma'}+\Delta_{\Gamma'}))\cong H^0(\Gamma, m(K_\Gamma+\Delta_\Gamma))$. 
\end{claim}

It is possible that two components of $D^n$ are conjugated under $\tau$ but do not belong to the same irreducible component of $S$. However, we have the following. 

\begin{claim}\label{claim:involution_and_exceptional_curve}
  A component $\Gamma\!$ of\, $D^n$ is $\varphi$-vertical if and only if $\tau(\Gamma)$ is. Moreover, $\Gamma$ is non--$\varphi$-vertical 
  if and only if $K_\Gamma+\Delta_\Gamma$ is ample.
\end{claim}

\begin{proof}
The one-dimensional component $\Gamma$ is $\varphi$-vertical if and only if $\psi(\Gamma)$ is a point. Moreover, $\psi(\Gamma)$ is a point if and only if $K_\Gamma+\Delta_\Gamma$ has Kodaira dimension zero. Since $\tau$ sends $K_\Gamma+\Delta_\Gamma$ to $K_{\tau(\Gamma)}+\Delta_{\tau(\Gamma)}$, we obtain the result.
\end{proof}

Next we gather some observations about the non--$\varphi$-vertical components of $D^n$ and their images in $T$. 

\begin{claim}\label{claim:horinzontal_components_in_big_case}
Let $(S_i,\Delta_i+D_i)$ be such that $\kappa_i=2$. Then $(T_i,(\varphi_i)_*(\Delta_i+D_i))$ is the log canonical model of $(S_i,\Delta_i+D_i)$. If\, $\Gamma$ is a non--$\varphi_i$-vertical irreducible component of $D_i$, then $\Gamma$ is the normalization of a component of $(\varphi_i)_*D_i$. 
\end{claim}

\begin{claim}\label{claim:horizontal_components}
Let $(S_i,\Delta_i+D_i)$ be such that $\kappa_i=1$, and assume that the non--$\varphi$-vertical sub-curve $\Theta$ of $D_i+\Delta_i^{=1}$ is non-empty. Then one of the following holds: 
	\begin{enumerate}
		\item  $\Theta$ is irreducible, and $\varphi_i|_{\Theta}$ is an isomorphism.
		\item  $\Theta$ is irreducible, and $\varphi_i|_{\Theta}$ has degree two.
		\item  $\Theta=\Gamma_1+\Gamma_2$, and each $\varphi_i|_{\Gamma_j}$ is an isomorphism.
	\end{enumerate}
\end{claim}

\begin{proof}
Let $\eta$ be the generic point of $T_i$. Then the generic fiber $F=S_i\times_T k(\eta)$ is a regular proper curve. Since $\varphi_*\sO_{S_i}=\sO_{T_i}$, we obtain $H^0(F,\sO_F)=k(\eta)$. We have
	$$\deg_{k(\eta)} K_F=\deg_{k(\eta)} (-\Delta_i-D_i)|_F\leq \deg_{k(\eta)}(-\Theta)|_F<0.$$
Thus $h^0(F,\omega_F)=h^1(F,\sO_F)=0$. By~\cite[Section~10.6]{Kollar_Singularities_of_the_minimal_model_program}, we deduce that $\deg_{k(\eta)} K_F=-2$. Hence $\deg_{k(\eta)}\Theta|_F\leq 2$, which means that $\Theta$ has at most two irreducible components. If $\Gamma_j$ is an irreducible component such that $\deg_{k(\eta)} \Gamma_j|_F=1$, then $\varphi|_{\Gamma_j}\colon \Gamma_j\to T$ is a finite birational morphism of normal curves, hence an isomorphism.
\end{proof}

\begin{claim}\label{claim:log_involution_over_sources}
In the situation of \Cref{claim:horizontal_components}, if\, $\Theta\to T_i$ is separable of degree two, then there exists a non-trivial log involution of\, $(\Gamma^n,\Diff_{\Gamma^n}(\Delta+D-\Gamma))$ over $T_i$.
\end{claim}

\begin{proof}
First assume that $\Theta=\Theta^n$ is irreducible. Then $\varphi_i|_\Theta$ is Galois of degree two and induces an involution $\xi\colon \Theta\cong \Theta$ over $T_i$. 
We claim that $\xi$ preserves the line bundle $\sO(m(K_\Theta+\Delta_\Theta))$. Indeed, fix any global meromorphic form $\omega\in H^0(S_i,\sO(m(K_{S_i}+D_i+\Delta_i)))$, and take $s\in H^0(T_i,\sO(H))$ such that $\varphi_i^*s=\omega$. Then since $\xi$ commutes with $\varphi_i|_{\Theta}$, we have 
	$$\omega|_{\Theta}=\left(\varphi_i|_{\Theta}\right)^*s=\left(\varphi_i|_{\Theta}\circ \xi\right)^*s=\xi^*\omega|_{\Theta}.$$
Since the global sections of $\sO(m(K_S+D+\Delta))$ generate, our claim is proved.

Now assume that $\Theta=\Gamma_1+\Gamma_2$ and that both $\varphi|_{\Gamma_i}$ are isomorphisms. Then one proves as above that
				$$\xi:=\varphi_i|_{\Gamma_2}^{-1}\circ \varphi_i|_{\Gamma_1}\colon (\Gamma_1,\Delta_{\Gamma_1})\longrightarrow (\Gamma_2,\Delta_{\Gamma_2})$$
		is a log isomorphism. 
\end{proof}

We are ready to show the finiteness of the relation generated by $(j,j')\colon E\rightrightarrows T$. Recall that $E$ is endowed with an involution $B^\tau$ defined in \Cref{claim:description_of_involution_downstairs} and that $j'=j\circ B^\tau$. If $\Gamma_{B^\tau}\hookrightarrow E\times_k E$ is the graph of the involution, the finite morphism $\Gamma_{B^\tau}\to T\times_k T$ induces an equivalence relation $R(B^\tau)\rightrightarrows T$ (see \Cref{section:Quotients_by_equivalence_relation}).

\begin{claim}\label{claim:finiteness}
The equivalence relation $R(B^\tau)$ defined by $(j,j')\colon E\rightrightarrows T$ is finite.
\end{claim}

\begin{proof}
We study the pullback of the equivalence relation $R(B^\tau)\rightrightarrows T$ through the finite structural morphism $j\colon E\to T$. If we can show that this equivalence relation on $E$ is finite, it will follow that the equivalence relation $R(B^\tau)\rightrightarrows T$ is finite. 
This pullback is the equivalence relation generated by two types of pre-relations on $E$:
	\begin{enumerate}
		\item the isomorphism $B^\tau\colon E\cong E$, 
		\item the fibers of $j\colon E\to T$. 
	\end{enumerate}
Outside a zero-dimensional closed subset $Z\subset j(E)$, the fibers of $j$ are either singletons or of order two. Indeed, over a dense open subset of $j(E)$, the morphism $j$ is either the normalization or the degree two morphism $\Theta\to T_i$, as in \Cref{claim:log_involution_over_sources}. The fibers of the degree two morphisms that are separable are the orbits of the log involutions $\xi\colon \Theta^n\cong \Theta^n$ described in that same claim. Under the identification $\psi(\Theta^n)=\Spec(H^0(\Theta^n, m(K_{\Theta^n}+\Delta_{\Theta^n})))$, we can describe this involution as follows: $\xi$ induces an automorphism $\xi^*$ of $H^0(\Theta^n,m(K_{\Theta^n}+\Diff_{\Theta^n}(\Delta+D-\Theta)))$, inducing in turn an automorphism $B^\xi$ of $\psi(\Theta^n)$. We can extend $B^\xi$ to an automorphism of the whole $E$ by declaring it to be the identity on the other components.
	
	Let us  first study the relation generated by the group of automorphisms $G=\langle B^\tau, \{B^\xi\}\rangle$ of $E$. We claim that $G\leq \Aut_k(E)$ is finite. Actually, the map 
		$$\Aut_k(D^n,\Diff_{D^n}\Delta)\supset G'=\langle \tau,\{\xi\}\rangle\longrightarrow G,\quad \phi\longmapsto B^\phi$$
	is a group morphism, and therefore it is surjective. So it suffices to show that $G'$ is finite. We have a natural group morphism $G'\to\Bij(\pi_0(D^n))$, and so we need only to show that its kernel $K$ is finite. Since by definition elements of $K$ send each component $\Gamma\in\pi_0(D^n)$ onto itself, we have a group morphism
		$$ \iota\colon K\longrightarrow\prod_{\Gamma\in\pi_0(D^n)}
		\Aut_k(\Gamma,\Delta_\Gamma)$$
which is injective. To conclude that $K$ is finite, we make the following observations:
	\begin{itemize}
		\item Let $\pr_\Gamma$ be the projection map from the product to its factor indexed by $\Gamma$. It is sufficient to show that each $(\pr_\Gamma\circ\iota)(K)$ is finite.
		\item If $\Gamma$ is non--$\varphi$-vertical, then $\Aut_k(\Gamma,\Delta_\Gamma)$ is finite by \Cref{proposition:log_auto_of_curves} since $K_\Gamma+\Delta_\Gamma$ is ample. Thus $(\pr_\Gamma\circ\iota)(K)$ is automatically finite.
		\item If $\Gamma$ is $\varphi$-vertical, then we claim that $(\pr_\Gamma\circ\iota)(K)$ has order at most two. Indeed, in \Cref{claim:horizontal_components}, we defined the isomorphisms $\xi$ between some non--$\varphi$-vertical curves, and we extend them to $D^n$ by the identity on the other components. Thus every $(\pr_\Gamma\circ\iota)(\xi)$ is the identity. As $K\subset G'=\langle \tau,\{\xi\}\rangle$, the image $(\pr_\Gamma\circ\iota)(K)$ is either trivial, or $\{\id, (\pr_\Gamma\circ\iota)(\tau)\}$ in case $\tau$ fixes $\Gamma$.
	\end{itemize}
	
	Now we must also declare to be equivalent those points that belong to the fiber above the points of $Z$; this means merging some $G$-orbits together. For the moment, it is sufficient to know that the set $Z$ is finite (we will describe it more precisely in \Cref{section:special_set} below). Since the new relations we must add on $E$ are supported on $G\cdot j^{-1}(Z)\times_k G\cdot j^{-1}(Z)$, which is finite over $k$, we obtain that the pullback of $R(B^\tau)$ on $E$ is finite.
\end{proof}

\subsubsection{Analysis of the special set}\label{section:special_set}
Before proceeding to the descent of the line bundle $H$, we study the special set $Z\subset$ considered during the proof of \Cref{claim:finiteness}. Recall that it is the finite set of those $z\in j(E)$ such that $j^{-1}(z)$ is not contained in a single $G$-orbit. 

\begin{claim}\label{claim:special_set_is_lc_strata}
Every point of\, $Z$ is the image of an lc center of\, $(S,\Delta+D)$.
\end{claim}

\begin{proof}
Let $z\in T_i$ be a point of $Z$. If $\dim T_i=0$, then the claim is clear. If $\dim T_i=1$, then by \Cref{claim:horizontal_components}, it follows that $z$ is the contraction of a $\varphi$-vertical component of $\Delta_i^{=1}+D_i$. 

Now assume that $\dim T_i=2$. If $z=\varphi(\Gamma)$ for some $\Gamma\in \pi_0(D^n)$, then we are done. So assume that there is no $\Gamma\in \pi_0(D^n)$ such that $\varphi(\Gamma)=z$. This implies that $z\in \varphi_*D_i$. By \Cref{claim:horinzontal_components_in_big_case}, around $z$ the morphism $j$ is just the normalization of $\varphi_*D_i$ . As $j$ is not an isomorphism over $z$ by assumption, it follows that $z$ is a singular point of $\varphi_*D_i$. We distinguish two cases:
	\begin{enumerate}
		\item $\varphi_i\colon S_i\to T_i$ \textit{is an isomorphism above} $z$:  Since $(S_i,\Delta_i+D_i)$ is dlt, we see that $z$ is a stratum of $\Delta_i^{=1}+D_i$.
		\item $\varphi_i\colon S_i\to T_i$ \textit{is not an isomorphism above} $z$: In particular, there is a proper curve $C\subset S_i$ that contracts to $z$. As 
			$$\varphi_i^*\left(K_{T_i}+\varphi_*(\Delta_i+D_i)\right)=K_{S_i}+\Delta_i+D_i,$$
		by~\cite[Theorem~2.31(2)]{Kollar_Singularities_of_the_minimal_model_program}, we see that $C$ belongs to the reduced boundary $\Delta_i^{=1}+D_i$ and hence is an lc center.
	\end{enumerate}
The proof of the claim is complete.
\end{proof}


\begin{claim}\label{claim:uniqueness_of_springs}
For each $z\in Z$, let $W_z$ and $W'_z$ be two lc centers of\, $(S,\Delta+D)$ that are minimal for the property that $\varphi(W_z)=z=\varphi(W'_z)$. Then there is a log isomorphism $(W_z,\Diff^*_{W_z}(\Delta+D))\cong (W'_z,\Diff^*_{W'_z}(\Delta+D))$.
\end{claim}

\begin{proof}
If $z\in T_i$ with $\dim T_i\leq 1$, we apply \Cref{proposition:useful_results_from_gluing_article} with $\Theta=D_i+\Delta_i^{=1}$ and $\Upsilon=\Delta_i^{<1}$. Hence from now on we assume that $\dim T_i=2$. The fiber $\varphi^{-1}(z)\cap (\Delta_i^{=1}+ D_i)$ is connected by~\cite[Proposition~2.36]{Kollar_Singularities_of_the_minimal_model_program}, and contains at least a strata of $\Delta_i^{=1}+D_i$ by \Cref{claim:special_set_is_lc_strata}. If it is zero-dimensional, then $W_z=W_z'$. Assume that $\varphi^{-1}(z)\cap (\Delta_i^{=1}+ D_i)$ is one-dimensional; let $\Theta=\bigcup_{j=1}^n\Theta_j$ be the sub-curve of $\Delta_i^{=1}+D_i$ that contracts to $z$. Then the minimal lc centers above $z$ are the minimal strata of $\Theta$. These strata are isomorphic by \Cref{lemma:zero_dim_strata_on_fibers}.
\end{proof}

\begin{claim}\label{claim:curves_in_fibers_in_big_case}
In the case $z\in T_i$ with $\dim T_i=2$, if a minimal lc center $W_z$ above $z$ is a curve, then it is the unique minimal lc center over $z$.
\end{claim}

\begin{proof}
Let us keep the notation of the proof of \Cref{claim:uniqueness_of_springs}. By the connectedness of $\varphi^{-1}(z)\cap (\Delta_i^{=1}+ D_i)$, we see that a minimal lc center is a curve if and only if $\varphi^{-1}(z)\cap (\Delta_i^{=1}+ D_i)$ is an irreducible curve.
\end{proof}

\begin{claim}\label{claim:genus_one_curves_in_fibers}
In the case $z\in T_i$ with $\dim T_i=1$, if one minimal $W_z$ above $z$ is a curve of genus one, then it is the unique minimal lc center over $z$.
\end{claim}

\begin{proof}
This follows immediately from \Cref{proposition:useful_results_from_gluing_article} with $\Theta=D_i+\Delta_i^{=1}$ and $\Upsilon=\Delta_i^{<1}$.
\end{proof}

\subsubsection{Descent}
To conclude, we must show that the Cartier divisor $H$ descends along the quotient $q\colon T\to T_0=T/R(B^\tau)$. There is a useful reduction step we will make.

\begin{claim}\label{claim:trivial_gluing}
To descend $H$, we may assume that $\dim T_i\geq 1$ for all $i$.
\end{claim}

\begin{proof}
Assume that we have found a line bundle $H_0$ on $T_0$ such that $(q^*H_0)|_{T_i}\cong H|_{T_i}$ for every $i\geq 1$ such that $\dim T_i\geq 1$. If $\dim T_j=0$, then we trivially have $(q^*H_0)|_{T_j}\cong \sO_{T_j}\cong H|_{T_i}$. So $H_0$ is the line bundle we are looking for.
\end{proof}

For the rest of the proof, we use the method of~\cite[proof of Theorem~5.38]{Kollar_Singularities_of_the_minimal_model_program} and keep the notation of \Cref{claim:finiteness}. 

Let $T_H:=\Spec_T(\sum_{r\geq 0}H^0(T,rH))$ be the total space of $H$, and similarly
		$$E_H:=T_H\times_T E=\bigsqcup_{\Gamma\in\pi_0(D^n)} \Spec_T\biggl( \sum_{r\geq 0} H^0\bigl(\Gamma, rm\left(K_\Gamma+\Delta_\Gamma\right)\bigr)\biggr)\overset{j_H}{\longrightarrow} T_H.$$
Since $\tau$ is a log isomorphism of $(D^n,\Diff_{D^n}\Delta)$, the involution $B^\tau\colon E\cong E$ lifts to an involution $B^\tau_H\colon E_H\cong E_H$. The Cartier divisor $H$ descends to the quotient $T/R(B^\tau)$ if the equivalence relation $R(B^\tau_H)\rightrightarrows T_H$ is finite (see~\cite[Proposition~9.48 and Section~9.53]{Kollar_Singularities_of_the_minimal_model_program}).

As in the proof of \Cref{claim:finiteness}, we consider the pullback of $R(B^\tau_H)\rightrightarrows T_H$ to $E_H$. It is generated by two types of pre-relations:
	\begin{enumerate}
		\item the fibers of the structural morphism $j_H\colon E_H\to T_H$, 
		\item the isomorphisms $B^\phi_H\colon E_H\cong E_H$ induced by the $B^\phi\colon E\cong E$, where $\phi\in G\subset \Aut_k(E)$. (More precisely, each $\phi\colon E\cong E$ induce an automorphism of the graded section ring
                  $$\bigoplus_{\Gamma\in \pi_0(D^n)} \sum_{r\geq 0}H^0\bigl(\Gamma,rm\left(K_{\Gamma}+\Delta_\Gamma\right)\bigr),$$ inducing in turn the automorphism $B^\phi_H$ of $E_H$.)
	\end{enumerate}
	We have seen in \Cref{claim:finiteness} that the group $G=\langle B^\tau,\{B^\xi\}\rangle\subset \Aut_k(E)$ is finite. Therefore, $G_H:=\langle B_H^\tau,\{B_H^\xi\}\rangle$ is also finite, and so the $G_H$-orbits on $T_H$ are finite. 
	
	Now we must take in account the fibers of $j_H\colon E_H\to T_H$ which are not single $G_H$-orbits. The new relations we get are supported on the fibers of $E_H\to T$ over the closed finite subset $Z\subset j(E)$. By \Cref{claim:uniqueness_of_springs}, we see that the new relations come from some isomorphisms between the $W_z$, inducing isomorphisms between the section rings of the Cartier divisors $m(K_S+\Delta+D)|_{W_z}$, where $W_z$ runs through the minimal lc centers of $(S,\Delta+D)$ over the points $z$ of $Z$. As in~\cite[proof of Theorem~5.38]{Kollar_Singularities_of_the_minimal_model_program}, it suffices to show that the pluricanonical representations
	\begin{equation}\label{eqn:pluricanonical_rep}
	\Aut_k\bigl(W_z,\Diff^*_{W_z}(\Delta+D)\bigr)\longrightarrow \GL_k H^0\left(W_z,m(K_{W_z}+\Diff^*_{W_z}(\Delta+D)\right)
	\end{equation}
are finite.

	If $z\in T_i$ with $\dim T_i=2$ and a minimal lc center $W_z$ is one-dimensional, then by \Cref{claim:curves_in_fibers_in_big_case} we see that $j^{-1}(z)$ is contained in a $G$-orbit. If $z\in T_i$ with $\dim T_i=1$ and a minimal lc center $W_z$ is a genus one curve, then by \Cref{claim:genus_one_curves_in_fibers} we see that $j^{-1}(z)$ is also contained in a $G$-orbit. Thus we may assume that the minimal lc centers $W_z$ above $z$ are either zero-dimensional or genus zero curves. 
	
	If $W_z$ is zero-dimensional, then it is the spectrum of a finite field extension of $k$, and thus $\Aut_k(W_z)$ is finite. If $W_z$ is a genus zero curve, then the finiteness of the pluricanonical representations \eqref{eqn:pluricanonical_rep} is proved in \Cref{prop:pluricanonical_rep_for_genus_zero}. This concludes the proof. \hfill\qedsymbol
	

	

\begin{remark}\label{remark:Kollars_theory_illustrated}
The discussion above is a simple illustration of some key features of Koll\'{a}r's theory of sources and springs for crepant log structures. 
	\begin{enumerate}
		\item In Koll\'{a}r's terminology, $\varphi\colon (S,\Delta+D)\to T$ is a crepant log structure, the components of $D$ are the sources (of their images in $T$), and the components of $E$ are the springs (of their images in $T$). 
		\item In \Cref{claim:log_involution_over_sources}, in case $\Theta=\Gamma_1+\Gamma_2$, the fact that $(\Gamma_1,\Delta_{\Gamma_1})$ and $(\Gamma_2,\Delta_{\Gamma_2})$ are log isomorphic to each other corresponds to the uniqueness of the source up to a crepant birational map; see \cite[Theorem-Defintion~4.45(1)]{Kollar_Singularities_of_the_minimal_model_program}.
		\item  In \Cref{claim:log_involution_over_sources}, in case $\Theta\to T_i$ is a separable double cover, the fact that the extension of function fields is Galois and that the Galois involution can be realized by a log automorphism of $(\Theta,\Delta_\Theta)$ corresponds to the Galois property of springs; see~\cite[Theorem-Definition~4.45(5)]{Kollar_Singularities_of_the_minimal_model_program}.
		\item In \Cref{claim:log_involution_over_sources}, in case $\Theta=\Gamma_1+\Gamma_2$, we have that $(S_i, \Gamma_1+\Gamma_2+(\Delta_i-\Theta))\to T$ is a weak $\bP^1$-link (see~\cite[Remark~5.2.10]{Posva_Gluing_for_surfaces_and_threefolds}). Indeed, the general fiber $F=\varphi_i^{-1}(t)$ is an integral Gorenstein proper curve of genus zero over $k(t)$ with $H^0(F,\sO_F)=k(t)$ and has an invertible sheaf $\Gamma_1|_F$ of degree $\Gamma_1\cdot_{k(t)}F=1$. Thus $F\cong \bP^1_{k(t)}$ by~\cite[0C6U]{Stacks_Project}.
		\item To show finiteness and descent, we have reduced both times to a question about representation of a group of log automorphisms of $E$ on the space of pluricanonical sections of $H$. This corresponds to the crucial role that pluricanonical representations have in~\cite[Theorem~5.36, Corollary~5.37 and Theorem~5.38]{Kollar_Singularities_of_the_minimal_model_program}. Our case is easily manageable since the groups of log automorphisms that appear are finite to begin with, so the finiteness of the representations is automatic.
	\end{enumerate}
\end{remark}

\subsection{Relative case}

We prove abundance in the relative setting. We deduce the relative version from the absolute version, following the strategy of~\cite{Tanaka_Abundance_for_lc_surfaces_over_imperfect_fields}.

\begin{assumption}\label{notation:relative_setting}
Let $(S_0,\Delta_0)$ be a slc surface and $f_0\colon S_0\to B_0$ a projective morphism, where $B_0$ is quasi-projective over a field $k$ of positive characteristic. We assume that $K_{S_0}+\Delta_0$ is $f_0$-nef.
\end{assumption}

We aim to show that $K_{S_0}+\Delta_0$ is $f_0$-semi-ample. 

\subsubsection{Reduction to separable nodes} As in \Cref{section:reduction_to_sep_nodes_in_abs_case}, we reduce to the case where $S_0$ has only separable nodes. The proof is similar, so we omit it.

\subsubsection{Reduction to the projective case}
Next we reduce to case of a projective base. 
Since $B_0$ is quasi-projective over $k$, it embeds as a dense open subset of a projective $k$-scheme $B$. We look for a projective slc compactification of $(S_0,\Delta_0)$ over $B$. There is a commutative diagram
	$$\begin{tikzcd}
	S_0\arrow[r, "j", hook]\arrow[d, "f_0"] & S\arrow[d, "f"] \\
	B_0\arrow[r, hook] & B\rlap{,}
	\end{tikzcd}$$
where $f$ is projective and $j$ is a dense open embedding. Since $S_0$ is already projective over $B_0$, we may assume that $f^{-1}(B_0)=S_0$. Using the method of~\cite[Theorem~3.7.1]{Posva_Gluing_for_surfaces_and_threefolds}, we may furthermore assume that the singular codimension one points of $S$ are contained in $S_0$. By our first reduction step, we may and will assume that these nodes are separable.


Let $\Delta$ be the closure in $S$ of $\Delta_0$, and let $(\bar{S},\bar{\Delta}+\bar{D})$ be the normalization of $(S,\Delta)$. By the hypothesis on~$S$, $\bar{D}$ is the closure of the conductor divisor of $S_0^n\to S_0$. Thus the normalization of its dense open subset $\bar{D}\cap S_0^n$ is equipped with an involution $\tau$. Since $\bar{D}^n$ is a normal projective curve, $\tau$ extends uniquely to an involution of $\bar{D}^n$.

Log resolution for normal excellent surfaces is known; see~\cite{Lipman_Resolution_of_surface_singularities, Artin_Lipman_proof_of_resolution_for_surfaces}. In particular, by repeatedly blowing up some closed points on the complement of $S_0^n$ and taking the strict transforms of our divisors, we may achieve the following:
	\begin{enumerate}
		\item The scheme $\bar{S}$ is regular at the points of the boundary $\partial\bar{S}:=\bar{S}\setminus S_0^n$.
		\item $\bar{\Delta}\cap \bar{D}$
                  is contained in $S_0^n$. 
		\item If $\bar{E}$ is the divisorial part of $Z$, then $\Supp(\bar{\Delta})+\bar{D}+\bar{E}$ is simple normal crossing in a neighborhood of~$Z$.
	\end{enumerate}
In particular, $(\bar{S},\bar{\Delta}+\bar{D})$ is lc. However, $K_{\bar{S}}+\bar{\Delta}+\bar{D}$ might not be nef over $B$. We can run a $(K_{\bar{S}}+\bar{\Delta}+\bar{D})$-MMP over $B$, but in order to denormalize, we ultimately want to recover an action of $\tau$ on the \emph{log pair} obtained using adjunction on the pushforward of $\bar{D}$. In particular, $\tau$ must preserve the different induced by the pushforward of $\bar{\Delta}$, and so it is crucial to control the intersection $\bar{D}\cap\bar{\Delta}$ along the process. It turns out that the $(K_{\bar{S}}+\bar{\Delta}+\bar{D})$-MMP over $B$ is good enough, but this requires a proof.

\begin{claim}\label{claim:pseudo_MMP_for_slc}
There exists a birational morphism over $B$
		$$\varphi\colon (\bar{S},\bar{\Delta}+\bar{D})\longrightarrow (\bar{S}',\bar{\Delta}'+\bar{D}'),$$
with $\bar{\Delta}'=\varphi_*\bar{\Delta}$ and $\bar{D}'=\varphi_*\bar{D}$, such that
	\begin{enumerate}
		\item\label{claim322-a} $(\bar{S}',\bar{\Delta}'+\bar{D}')$ is lc and $K_{\bar{S}'}+\bar{\Delta}'+\bar{D}'$ is nef over $B$;
		\item\label{claim322-b} $\varphi\colon S_0^n\to \varphi(S_0^n)$ is an isomorphism of open subsets;
		\item\label{claim322-c} $\bar{D}^n\cong (\bar{D}')^n$, so we can transport the involution $\tau$ to $(\bar{D}')^n$;
		\item\label{claim322-d} the induced $R(\tau)\rightrightarrows \bar{S}'$ is a finite equivalence relation; and 
		\item\label{claim322-e} $\tau$ preserves the pullback of $K_{\bar{S}'}+\bar{\Delta}'+\bar{D}'$ to $(\bar{D}')^n$.
	\end{enumerate}
\end{claim}

\begin{proof}
Consider an irreducible curve $C\subset \bar{S}$ vertical over $B$, and assume that $C$ is a $(K_{\bar{S}}+\bar{\Delta}+\bar{D})$-negative extremal ray. We make the following observations:
	\begin{itemize}
		\item Since $K_{\bar{S}}+\bar{\Delta}+\bar{D}$ is nef over $B_0$, the curve $C$ belongs to the boundary $\partial\bar{S}$.
		\item Since $C$ is not a component of $\bar{\Delta}+\bar{D}$, we have $K_{\bar{S}}\cdot C<0$ (the intersection makes sense since $\bar{S}$ is regular in a neighborhood of $C$).
		\item By~\cite{Tanaka_MMP_for_excellent_surfaces}, we can contract $C$, and since $K_{\bar{S}}+\bar{\Delta}+\bar{D}$ is pseudo-effective, we obtain a birational contraction. Thus $C^2<0$ by~\cite[Section~10.1]{Kollar_Singularities_of_the_minimal_model_program}.
	\end{itemize}
By \Cref{proposition:Castelnuovo_criterion}, there is a proper birational morphism $\varphi_1\colon \bar{S}\to\bar{S}_1$ to a projective $B$-surface such that $\varphi_1(C)$ is a regular point of $\bar{S}_1$ and $\varphi_1\colon \bar{S}\setminus C\to \bar{S}_1\setminus$ is an isomorphism. Define $\bar{\Delta}_1=(\varphi_1)_*\bar{\Delta}$ and $\bar{D}_1=(\varphi_1)_*\bar{D}$. We claim that $(\bar{S}_1,\bar{\Delta}_1+\bar{D}_1)$ is lc: this is because $\varphi_1$ is a step of the $(K_{\bar{S}}+\bar{\Delta}+\bar{D})$-MMP.

Assume that $C$ intersects $\bar{D}$. Letting $K=H^0(C,\sO_C)$, we have by \Cref{proposition:Castelnuovo_criterion} that
		\begin{equation*}
		(K_{\bar{S}}+\bar{\Delta}+\bar{D})\cdot_K C = (\bar{\Delta}+\bar{D})\cdot_K C +K_{\bar{S}}\cdot_KC \geq 1 - 1 = 0,
		\end{equation*}
which gives a contradiction with the fact that $C$ was chosen to be $(K_{\bar{S}'}+\bar{\Delta}'+\bar{D}')$-negative. So $C$ does not intersect $\bar{D}$, and thus $\varphi_1(C)\notin\bar{D}_1$. In other words, $\varphi_1$ is an isomorphism in a neighborhood of $\bar{D}$. From this it follows immediately that properties \eqref{claim322-c}, \eqref{claim322-d} and \eqref{claim322-e} are satisfied on $\bar{S}_1$. Since $C\subset \partial\bar{S}$, we also see that property \eqref{claim322-b} holds for $\varphi_1$.



This shows that we can run a $(K_{\bar{S}}+\bar{\Delta}+\bar{D})$-MMP over $B$ while retaining the properties \eqref{claim322-b}, \eqref{claim322-c}, \eqref{claim322-d} and~\eqref{claim322-e}. Since $K_{\bar{S}}+\bar{\Delta}+\bar{D}$ is pseudo-effective over $B$, the MMP stops with a relative minimal model $(\bar{S}',\bar{\Delta}'+\bar{D}')$, which is the one we are looking for.
\end{proof}

By the two last items of \Cref{claim:pseudo_MMP_for_slc} and by~\cite[Theorem 1]{Posva_Gluing_for_surfaces_and_threefolds}, we can de-normalize the pair $(\bar{S}', \bar{\Delta}'+\bar{D}')$ along the gluing data given by $R(\tau)\rightrightarrows \bar{S}'$. We obtain an slc surface pair $(S',\Delta')$ which contains $(S_0,\Delta_0)$ as an open subset and is such that $K_{S'}+\Delta'$ is nef over $B$. The proof that $S'$  is projective over $B$ is similar to the analog statement in~\cite[Theorem~3.7.1]{Posva_Gluing_for_surfaces_and_threefolds}.

\begin{claim}
With the notation as above, it is sufficient to show that $K_{S'}+\Delta'$ is semi-ample over $B$.
\end{claim}

\begin{proof}
Since $(S',\Delta')\times_BB_0=(S,\Delta)$, the claim holds because relative semi-ampleness is local on the base (and more generally preserved by base-change; see~\cite[Lemma~2.12]{Cascini_Tanaka_Relative_semi_ampleness_in_pos_char}).
\end{proof}


\subsubsection{Conclusion of the proof}\label{section:relative_to_absolute}
We assume from now on that $B_0$ is projective over $k$ and $(S_0,\Delta_0)$ slc projective over $B_0$ with normalization $(\bar{S}_0,\bar{\Delta}_0+\bar{D})$. Let $A$ be an Cartier ample divisor on $B_0$. By~\cite[Proposition~4.11]{Tanaka_Abundance_for_lc_surfaces_over_imperfect_fields}, the divisor $K_{\bar{S}_0}+\bar{\Delta}_0+\bar{D}+n^*f_0^*(mA)$ is nef for $m$ large enough. Thus $K_{S_0}+\Delta_0+f_0^*(mA)$ is nef, hence semi-ample over $k$ by \Cref{theorem:abundance_absolute_case}. In particular, it is $f_0$-semi-ample; say that $L:=\sO(r(K_{S_0}+\Delta_0+f_0^*(mA)))$ defines a morphism over $B_0$. Tensoring the surjection $f_0^*(f_0)_*L\to L$ by $f_0^*\sO(-mA)$, we see that $K_{S_0}+\Delta_0$ is $f_0$-semi-ample. 

This completes the proof of \Cref{theorem:abundance_relative}. \hfill\qedsymbol

\section{Relative abundance for one-parameter families of surfaces}\label{section:mixed_characteristic}
In this section, we work over an excellent regular one-dimensional scheme $S$ that is separated and essentially of finite type over $\Spec(\bZ)$. If the structural morphism $S\to \Spec(\bZ)$ is not constant, we say that $S$ is of \emph{mixed characteristic}. If the structural morphism has constant value $p\geq 0$, we say that $S$ is of \emph{equicharacteristic} $p$. (With this terminology, it may happen that $S$ is of mixed characteristic while some component of $S$ is of equicharacteristic.)

\begin{observation}\label{observation:mixed_char_base}
With the notation as above, assume that $S$ is irreducible of mixed characteristic. \emph{Then the only point of\, $S$ with residue characteristic zero is the generic point of $S$.} Indeed, let $j\colon S\to \Spec(\bZ)$ be the structural morphism. If $s\in S$ is a specialization of $\eta\in S$, then $j(s)$ is a specialization of $j(\eta)$ (simply because $j$ is continuous). Since $j$ is not constant, this implies that the generic point of $S$ maps to the generic point of $\Spec(\bZ)$. If $S$ is the spectrum of a DVR, then our claim follows at once. So assume that $S$ is of finite type over $\Spec(\bZ)$. Since $j$ is not constant, it must be flat, and since $\dim S=1$, we see that the fibers of $j$ are zero-dimensional. Thus $j$ is quasi-finite. Using Zariski's main theorem for quasi-finite morphisms (see~\cite[05K0, 0F2N]{Stacks_Project}), we see that if $s\in S$ is a closed point, then $j(s)\in\Spec(\bZ)$ is a closed point. Therefore, closed points of $S$ cannot be sent to the generic point of $\Spec(\bZ)$.
\end{observation}

\begin{lemma}\label{lemma:slc_adjunction}
Let $S$ be an excellent regular one-dimensional scheme and $X\to S$ an equidimensional separated morphism of finite type with $S_2$ fibers. Assume that $\Delta$ is a $\bQ$-divisor on $X$ such that $(X,\Delta+X_t)$ is slc for every closed point $t\in S$. Then: 
	\begin{enumerate}
		\item Every fiber $X_s$ is demi-normal, and there is a boundary $\Delta_s$ on $X_s$ such that $K_X+\Delta+X_s|_{X_s}= K_{X_s}+\Delta_s$.
		\item Every pair $(X_s,\Delta_s)$ is slc.
	\end{enumerate}
\end{lemma}

\begin{proof}
Fix a point $s\in S$. If $s$ is a generic point, the result is obtained by localization. So assume that $s\in S$ is closed. The scheme $X_s$ is $S_2$ by assumption, and it is generically reduced since the divisor $\Delta+X_s$ has coefficients at most one. So $X_s$ it is reduced. By~\cite[Theorem~2.31]{Kollar_Singularities_of_the_minimal_model_program}, it is at worst nodal in codimension one. Therefore, $X_s$ is demi-normal. In particular, we can perform adjunction on $X_s$ (see~\cite[Definition~4.2]{Kollar_Singularities_of_the_minimal_model_program}), and we obtain a boundary $\Delta_s$ on $X_s$ satisfying $K_X+\Delta+X_s|_{X_s}=K_{X_s}+\Delta_s$. 

Let $(\bar{X},\bar{\Delta}+\bar{D}+\bar{X}_s)$ be the normalization of $(X,\Delta+X_s)$. The latter pair is slc by assumption, so the former pair is lc. If $\bar{X}_s^n$ is the normalization of $\bar{X}_s$, by adjunction $(\bar{X}_s^n,\Diff(\bar{\Delta}+\bar{D}))$ is lc. Thus we have a commutative diagram
		$$\begin{tikzcd}
		\left(\bar{X}_s^n,\Diff(\bar{\Delta}+\bar{D})\right) \arrow[d, "\mu"] \arrow[r] & \left(\bar{X},\bar{\Delta}+\bar{D}+\bar{X}_s\right) \arrow[d] \\
		\left(X_s,\Delta_s\right)\arrow[r] & \left(X,\Delta+X_s\right)\rlap{,}
		\end{tikzcd}$$
where every arrow different from $\mu$ is a finite crepant morphism. Therefore, $\mu$ is also finite crepant.

We claim that $\mu$ is the normalization morphism. Since $X_s$ is a reduced hypersurface in $X$, we see that $X$ is regular at the generic points of $X_s$. Thus $\bar{X}\to X$ is an isomorphism at the generic points of $\bar{X}_s$, and it follows that $\bar{X}_s\to \bar{X}$ is finite surjective and birational. This implies that $\mu$ is indeed the normalization morphism of $X_s$. Since $(\bar{X}_s^n,\Diff(\bar{\Delta}+\bar{D}))$ is lc, we obtain by definition that $(X_s,\Delta_s)$ is slc.
\end{proof}

\begin{theorem}\label{theorem:abundance_in_mixed_characteristic}
Let $S$ be an excellent regular one-dimensional scheme that is separated and of finite type over $\Spec(\bZ)$. Assume that no component of $S$ is of equicharacteristic zero. Let $f\colon (X,\Delta)\to S$ be a surjective flat projective morphism of relative dimension two with $S_2$ fibers. Let $\Delta$ be a $\bQ$-divisor on $X$ such that $(X,\Delta+X_s)$ is slc for every closed point $s\in S$.

If $K_X+\Delta$ is $f\!$-nef, then it is $f\!$-semi-ample.
\end{theorem}

\begin{proof}
For any point $s\in S$ (not necessarily closed), it follows from \Cref{lemma:slc_adjunction} that the fiber $(X_s,\Delta_s)$ is a projective slc surface over $k(s)$ with $K_{X_s}+\Delta_s$ nef. Using either \Cref{theorem:abundance_absolute_case} or~\cite[Theorem 8.5]{Flips_and_abundance_for_3folds}, we obtain that $K_{X_s}+\Delta_s$ is semi-ample. Therefore, the restriction of $K_X+\Delta$ to each fiber is semi-ample. We need to prove that it is semi-ample over $S$.

Since $S$ is the disjoint union of its irreducible components, we reduce to the case where $S$ is irreducible of equicharacteristic $p>0$ or irreducible of mixed characteristic. In the first case, the result follows from~\cite[Theorem 1.1]{Cascini_Tanaka_Relative_semi_ampleness_in_pos_char}. In the second case, we see by \Cref{observation:mixed_char_base} that the generic point of $S$ is its unique point with residue characteristic zero, and therefore the result follows from~\cite[Theorem 1.2]{Witaszek_Relative_semiampleness_in_mixed_char}.
\end{proof}

\section{Applications to threefolds}

\subsection{Reduced boundary of dlt pairs}


\begin{lemma}\label{lemma:boundary_of_3fold_dlt}
Let $(X,\Delta)$ be a $\bQ$-factorial dlt threefold defined over an arbitrary field $k$ of characteristic $p>5$. If $E\subset \Delta^{=1}$ is an irreducible component, then $E$ is normal.
\end{lemma}

\begin{proof}
The first paragraph of the proof of~\cite[Theorem~4.16]{Kollar_Singularities_of_the_minimal_model_program} works over any field and shows that the lc centers of a dlt pair are among the strata of its reduced boundary. Therefore, $(X,E)$ is plt in a neighborhood of $E$. Then we apply~\cite[Corollary~7.17]{Bhatt&Co_MMP_for_3folds_in_mixed_char} to obtain that $E$ is normal.
\end{proof}



\begin{lemma}\label{lemma:reduced_boundary_is_almost_demi_normal}
Let $(X,\Delta)$ be a $\bQ$-factorial dlt threefold defined over an arbitrary field $k$ of characteristic $p>5$. Let $\pi\colon Z\to \Delta^{=1}$ be the $S_2$-fication. Then $\pi$ is a finite universal homeomorphism and $Z$ is demi-normal.
\end{lemma}

\begin{proof}
The $S_2$-fication is characterized in \Cref{proposition:S_2_fication}. In particular, $\pi$ is finite and an isomorphism above the codimension one points. Thus by~\cite[Theorem~2.31]{Kollar_Singularities_of_the_minimal_model_program} and the fact that $Z$ is $S_2$, we obtain that $Z$ is demi-normal.

It remains to show that $\pi$ is a universal homeomorphism. It is surjective since it factors the normalization, and universally closed since it is finite. Thus it suffices to show that $\pi$ is injective on geometric points. This is proved as in~\cite[Proposition~5.1]{Waldron_MMP_for_3_folds_in_char_>5}, using \Cref{lemma:boundary_of_3fold_dlt} instead of~\cite[Proposition~2.11]{Waldron_MMP_for_3_folds_in_char_>5}.
\end{proof}

\begin{theorem}\label{theorem:restriction_to_boundary_is_semi_ample}
Let $(X,\Delta)$ be a projective $\bQ$-factorial dlt threefold over an arbitrary field $k$ of characteristic $p>5$. Assume that $K_X+\Delta$ is nef. Then $(K_X+\Delta)|_{\Delta^{=1}}$ is semi-ample.
\end{theorem}

\begin{proof}
Let $\pi\colon Z\to \Delta^{=1}$ be the $S_2$-fication. Then $Z$ is demi-normal by \Cref{lemma:reduced_boundary_is_almost_demi_normal}; by adjunction, the pair $(Z,K_Z+\Diff_Z\Delta^{<1})$ is a projective slc surface and $K_Z+\Diff_Z\Delta^{<1}=(K_X+\Delta)|_Z$ is nef. Then by \Cref{theorem:abundance_absolute_case}, $(K_X+\Delta)|_Z$ is semi-ample. By \Cref{lemma:reduced_boundary_is_almost_demi_normal}, the morphism $\pi$ is a universal homeomorphism; hence it factors a $k$-Frobenius of $Z$, and we deduce that $(K_X+\Delta)|_{\Delta^{=1}}$ is also semi-ample; see~\cite[Lemma~2.11(iii)]{Cascini_Tanaka_Relative_semi_ampleness_in_pos_char}.
\end{proof}

\subsection{Good minimal models}
\Cref{theorem:restriction_to_boundary_is_semi_ample} can be used to establish abundance for log minimal threefold pairs of maximal Kodaira dimension. We follow the arguments of~\cite{Waldron_MMP_for_3_folds_in_char_>5}, with  special attention to the hypothesis on the base-field.

\begin{theorem}\label{thm:abundance_in_dim_3}
Let $(X,\Delta)$ be a three-dimensional lc projective pair over a field $k$ of characteristic $p>5$ which is $F$-finite $($i.e.,~$[k:k^p]<\infty)$. Let $f\colon X\to S$ be a projective morphism onto a projective normal $k$-scheme satisfying $f_*\sO_X=\sO_S$. If\, $K_X+\Delta$ is $f\!$-nef and $f\!$-big, then it is $f\!$-semi-ample.
\end{theorem}

\begin{proof}
The proof of~\cite[Theorem 1.1]{Waldron_MMP_for_3_folds_in_char_>5} carries over to our situation with little changes. We only explain the necessary modifications.

The case where $S$ is the spectrum of a field is proved in~\cite[Section~6.1, Steps 1-4]{Waldron_MMP_for_3_folds_in_char_>5}. The necessary ingredients of the proof are: existence of $\bQ$-factorial crepant dlt models, some results from the MMP for three-dimensional klt pairs (cone theorem, base-point-freeness, existence of log minimal models), abundance for slc surface pairs, Keel's theorem and~\cite[Proposition~2.14(1) and Lemma~2.16]{Waldron_MMP_for_3_folds_in_char_>5}. They hold in our setup:
	\begin{enumerate}
		\item $\bQ$-factorial crepant dlt models exist in characteristic $>5$ by~\cite[Corollary~9.19 and Remark~9.20]{Bhatt&Co_MMP_for_3folds_in_mixed_char}.
		\item For three-dimensional klt pairs over $F$-finite fields of characteristic $p>5$, the cone theorem, the base-point-freeness theorem and the existence of log minimal models are all established in~\cite{Das_Waldron_MMP_for_3folds_over_imperfect_fields}.
		\item Abundance for slc surface pairs over arbitrary fields is \Cref{theorem:abundance_absolute_case}.
		\item Keel's theorem holds over any field of positive characteristic (see~\cite[Proposition~2.2]{Waldron_MMP_for_3_folds_in_char_>5}).
		\item Lemma~2.16 of \cite{Waldron_MMP_for_3_folds_in_char_>5} relies on~\cite[Theorem~1.7]{Waldron_MMP_for_3_folds_in_char_>5} and~\cite[Proposition~2.14(1)]{Waldron_MMP_for_3_folds_in_char_>5}. The last two results have analogs over fields of characteristic $p>5$ by~\cite[Theorem H]{Bhatt&Co_MMP_for_3folds_in_mixed_char} and~\cite[Theorem G]{Bhatt&Co_MMP_for_3folds_in_mixed_char}, respectively. (See~\cite[Theorem~1.1]{Das_Waldron_MMP_for_3folds_over_imperfect_fields} and~\cite[Proposition~4.10]{Das_Waldron_MMP_for_3folds_over_imperfect_fields} for the $F$-finite case.)
	\end{enumerate}
This suffices for the case where $S$ is a point. 

When $S$ is not a point, we argue as in~\cite[End of Section~6.1]{Waldron_MMP_for_3_folds_in_char_>5}. Here we use a result from~\cite{Tanaka_Semiample_perturbations_of_lc_pairs}, where it is assumed that the base-field contains an infinite perfect field. Let us explain how we can reduce to this situation, therefore concluding the proof. (See~\cite[Proof of Theorem~4.12]{Tanaka_Abundance_for_lc_surfaces_over_imperfect_fields} for a similar argument.) Let $K$ be the composite field of $k$ and $\bar{\bF}_p$. Since relative semi-ampleness descends faithfully flat morphisms, see \cite[Lemma~2.12(ii)]{Cascini_Tanaka_Relative_semi_ampleness_in_pos_char}, it is sufficient to prove that $K_X+\Delta$ is $f\!$-semi-ample after extending the scalars to $K$. Since $k\subset K$ is a separable extension, the hypotheses of the theorem survive the base-change, except possibly the $F$-finiteness of the base-field. Let $k'$ be an intermediate extension of $\bar{\bF}_p\subset K$ such that 
	\begin{enumerate}
		\item $k'$ is finitely generated over $\bar{\bF}_p$, and
		\item $(X_K,\Delta_K)$ and $f_K$ are defined over $k'$.
	\end{enumerate}
Since relative semi-ampleness is preserved by base-change, it is sufficient to prove the theorem over $k'$. The point is that $k'$ is finitely generated over $\bar{\bF}_p$, so it is $F$-finite, and therefore we can argue as in \cite[End of Section~6.1]{Waldron_MMP_for_3_folds_in_char_>5}.
\end{proof}

\begin{remark}
In \Cref{thm:abundance_in_dim_3}, the only place where the $F$-finiteness of the base-field is used is to apply the fundamental theorems of the MMP for klt threefolds proved in~\cite{Das_Waldron_MMP_for_3folds_over_imperfect_fields}. As the authors note in~\cite[Remarks 1.9 and 6.19]{Das_Waldron_MMP_for_3folds_over_imperfect_fields}, this is only assumed to ensure the existence of projective log resolutions in dimension three, which is not known yet over arbitrary fields (see~\cite[Section~2.2.1]{Das_Waldron_MMP_for_3folds_over_imperfect_fields} and~\cite[Section~2.3]{Bhatt&Co_MMP_for_3folds_in_mixed_char}). In particular, \Cref{thm:abundance_in_dim_3} holds over a field of characteristic $p>5$ if we assume the existence of projective log resolutions over that field.
\end{remark}

\newcommand{\etalchar}[1]{$^{#1}$}
\providecommand{\bysame}{\leavevmode\hbox to3em{\hrulefill}\thinspace}
\providecommand{\MR}{\relax\ifhmode\unskip\space\fi MR }
\providecommand{\MRhref}[2]{%
  \href{http://www.ams.org/mathscinet-getitem?mr=#1}{#2}
}
\providecommand{\href}[2]{#2}

\end{document}